\documentclass[10pt,twoside,english]{amsart}
\normalsize
\advance\oddsidemargin by -1.8cm
\advance\evensidemargin by -1.6cm
\usepackage{srcltx} 

\usepackage{amssymb}
\usepackage{babel}
\usepackage{amsmath}
\usepackage{amscd}   
\usepackage{epsfig}
\usepackage{rotating}
\usepackage{psfrag}
\usepackage[all]{xy}
\usepackage{multirow}
\usepackage{arydshln}
\usepackage{float}
\usepackage{enumerate}
\let\mathg\mathfrak
\theoremstyle{plain}
\newtheorem{cor}{Corollary}[section]
\newtheorem{lem}[cor]{Lemma}
\newtheorem{thm}[cor]{Theorem}

\newtheoremstyle{thmstylenn}
{15pt}
{5pt}
{\it}
{}
{\bf}
{ \ref{lem:sp2rep}.}
{ }
{}
\theoremstyle{thmstylenn}

\theoremstyle{definition}
\newtheorem{exa}[cor]{Example}
\newtheorem{NB}[cor]{Remark}
\newtheorem{dfn}[cor]{Definition}

\newcommand{\Kommentar}[1]{}

\newcommand{\bdm}{\begin{displaymath}}
\newcommand{\edm}{\end{displaymath}}
\newcommand{\be}{\begin{equation}}
\newcommand{\ee}{\end{equation}}
\newcommand{\ba}[1]{\begin{array}{#1}}
\newcommand{\ea}{\end{array}}
\newcommand{\bea}[1][]{\begin{eqnarray#1}}
\newcommand{\eea}[1][]{\end{eqnarray#1}}
\newcommand{\btab}{\begin{tabular}}
\newcommand{\etab}{\end{tabular}}

\usepackage{rotating}

\newcommand{\x}{\times}

\newcommand{\ra}{\rightarrow}

\newcommand{\Id}{\ensuremath{\mathrm{Id}}}

\newcommand{\del}{\partial}

\newcommand{\C}{\ensuremath{\mathbb{C}}}

\newcommand{\R}{\ensuremath{\mathbb{R}}}

\newcommand{\Scal}{\ensuremath{\mathrm{Scal}}}

\newcommand{\Cl}{\ensuremath{\mathcal{C}}}

\newcommand{\SL}{\ensuremath{\mathrm{SL}}}

\newcommand{\SU}{\ensuremath{\mathrm{SU}}}
\newcommand{\U}{\ensuremath{\mathrm{U}}}

\newcommand{\Sympl}{\ensuremath{\mathrm{Sp}}}

\newcommand{\so}{\ensuremath{\mathg{so}}}
\renewcommand{\sl}{\ensuremath{\mathg{sl}}}
\newcommand{\SO}{\ensuremath{\mathrm{SO}}}

\newcommand{\Sp}{\ensuremath{\mathrm{Sp}}}
\newcommand{\Gl}{\ensuremath{\mathrm{Gl}}}
\newcommand{\Spin}{\ensuremath{\mathrm{Spin}}}
\newcommand{\spin}{\ensuremath{\mathg{spin}}}

\newcommand{\g}{\ensuremath{\mathfrak{g}}}

\begin{document}
\def\haken{\mathbin{\hbox to 6pt{%
                 \vrule height0.4pt width5pt depth0pt
                 \kern-.4pt
                 \vrule height6pt width0.4pt depth0pt\hss}}}
    \let \hook\intprod
\setcounter{equation}{0}
%
%
\thispagestyle{empty}
%
\date{\today . MSC 2010:  53 C 25-29; 58 J 50; 58 J 60}
\title[Cones of $G$ manifolds and  Killing spinors with skew torsion
]{
Cones of $G$ manifolds and  Killing spinors with skew torsion}
%
%
%
\author{Ilka Agricola}
\author{Jos H\"oll}
\address{\hspace{-5mm} 
Ilka Agricola, Jos H\"oll \newline
Fachbereich Mathematik und Informatik \newline
Philipps-Universit\"at Marburg\newline
Hans-Meerwein-Strasse \newline
D-35032 Marburg, Germany\newline
{\normalfont\ttfamily agricola@mathematik.uni-marburg.de}\newline
{\normalfont\ttfamily hoellj@mathematik.uni-marburg.de}}

%
%
\begin{abstract}
This paper is devoted to the  systematic investigation
of the cone construction for Riemannian 
$G$ manifolds $M$, endowed with an invariant metric
connection with skew torsion $\nabla^c$, a `characteristic connection'. 
We  show how to define a $\bar G$ structure
on the cone $\bar M=M\x \R^+$ with a cone metric, and we 
prove that a Killing spinor with torsion on $M$ induces a  
spinor on $\bar M$ that is parallel w.\,r.\,t.~the characteristic connection
of the $\bar G$ structure. We establish the explicit
correspondence between classes of metric almost contact structures
on $M$ and almost hermitian classes on $\bar M$, resp. between classes
of $G_2$ structures on $M$ and $\Spin(7)$ structures on $\bar M$.   
Examples illustrate how this `cone correspondence
with torsion' works in practice.
\end{abstract}
\maketitle
\pagestyle{headings}
%
%
%
\section{Preliminaries}\noindent
%
\subsection{Introduction}
Given a complete Riemannian spin manifold $(M,g)$, the two most basic equations
that a spinor field $\psi$ can fulfill are the parallelism equation and the
Killing equation, 
\bdm
\nabla^g\psi\ =\ 0,\qquad 
\nabla^g_X\psi\ =\ \mu X\cdot\psi \text{  for some } \mu \in\R-\{0\},
\edm
where $\nabla^g$ denotes the Levi-Civita connection. Berger's holonomy theorem
yielded that the Ricci-flat manifolds with reduced Riemannian
holonomy $\SU(n),\ \Sp(n),\ G_2,$ or $\Spin(7)$ were candidates for manifolds
with parallel spinors, and indeed Wang proved in 1989 that these are
the only manifolds admitting parallel spinors, and determined the dimension
of the space of parallel spinors \cite{W89}. The geometric meaning of
the Killing equation stems from the fact that Riemannian Killing spinors
realize the equality case in Friedrich's seminal estimate of the first
eigenvalue of the Riemannian Dirac operator on  compact Riemannian manifolds 
of positive  curvature \cite{Friedrich80}. Independently, the Killing equation
was investigated in theoretical physics for supergravity theories in
dimensions $10$ and $11$ \cite{Duff84} and certain applications in
general relativity \cite{Penrose83}.
The first non-trivial compact examples of Riemannian manifolds with Killing 
spinors were found in dimensions $5\leq n\leq 7$ in 1980-1986 
(\cite{Friedrich80}, \cite{Friedrich&G85}, \cite{Duff86}). The link to 
non-integrable geometry and $G$ structures was established shortly after; 
for instance, a compact, connected and simply connected $6$-dimensional 
Hermitian manifold is nearly K\"ahler if and only if it admits 
a Riemannian Killing spinor \cite{Grunewald90}.
Similar results hold for Einstein-Sasaki structures in all odd dimension and
nearly parallel $G_2$-manifolds in dimension $7$ (\cite{FK89}, \cite{FK}).

The  connection between these two spinorial field equations was recognized
by Bryant in 1987, who proved that the cone over the nearly K\"ahler
manifold $\SU(3)/T^2$ was an integrable $G_2$ manifold, and that
the cone over the nearly parallel $G_2$ manifold $\SO(5)/\SO(3)$ was
an integrable $\Spin(7)$ manifold \cite{Bryant87}. 
B\"ar generalized this idea in 1993, 
he proved that the cone $(\bar M=M\x \R^+,  \bar g= r^2 g+dr^2)$ of a (compact)
Riemannian spin manifold $(M,g)$ with Riemannian Killing spinors is a (non-compact)
Ricci-flat Riemannian spin manifold with $\nabla^{\bar g}$-parallel spinors.
We will loosely call this phenomenon the \emph{cone correspondence}.
By combining this cone correspondence with Wang's classification result, one
obtains a complete overview about all geometries that can carry Riemannian 
Killing spinors: together with results by Hijazi \cite{Hijazi86}, the general picture
is basically that the non-integrable geometries listed above are, beside spheres, the 
\emph{only} possible ones. A great deal of effort has been invested in the
actual construction of such  non-integrable geometries. But while
there is a rich supply of  non-homogeneous Einstein-Sasaki
manifolds (see \cite{GMSW04}, \cite{BG}, and many others) 
and nearly parallel $G_2$ manifolds, 
compact nearly K\"ahler manifolds
have resisted so far all construction efforts in the non-homogeneous case,
though they are generally believed to exist.

Since then, there has been a lot of progress on $G$ structures on Riemannian
manifolds and general holonomy theory. Einstein-Sasaki manifolds,  
nearly K\"ahler $6$-manifolds, and
nearly parallel $G_2$-manifolds are only special instances of more general
Riemannian manifolds with structure group $\U(n)$, $\SU(3)$, or $G_2$. These can be
neatly divided in different classes, first through the study of their
characterizing differential equations (\cite{CG90}, \cite{CM92}, \cite{FG82}, 
\cite{F85}, \cite{GH80}), later by more general concepts like intrinsic torsion
(\cite{Salamon89}, \cite{Swann00}) and, closely related,  characteristic 
connections -- these are, by definition, invariant metric connections
with skew torsion (\cite{Fri2}, \cite{A06}). The integrable geometries covered by
Berger's theorem correspond to the  `trivial' class (though, of course, they
are highly non trivial objects). Many examples of different classes 
were constructed and
their special properties investigated in the last decades. 
As a common feature,
a certain, well-understood subclass of every possible $G$ structure
admits a unique $G$ invariant metric connection with skew torsion,
the \emph{characteristic connection} $\nabla^c$, and it induces 
(with a $1/3$ rescaling)
a characteristic Dirac operator that generalizes the Dolbeault operator on
Hermitian manifolds  and Kostant's `cubic' Dirac operator on naturally
reductive  homogeneous spaces (\cite{Bismut}, \cite{Agricola&F04a}, \cite{ABK12}). 

Again, a big incentive to study $G$ manifolds admitting a characteristic
connection came from theoretical physics, more precisely from superstring
 theory, where the characteristic torsion (by definition, it is a $3$-form
on the manifold) is interpreted as a higher order flux  (see
\cite{Strominger86}, \cite{Green&S&W87} for the first publications on the
topic; for more details, we refer to the vast literature on string 
compactifications). Spinor fields satisfying a generalized kind of
Killing\,/\,parallelism equation with torsion (the precise equation depends on
the model) are identified with supersymmetry
transformations. More recently, connections with skew torsion and their
 Dirac operators are also considered for the spectral action principle
and hypothetical applications in cosmic topology \cite{Lai&T12}.

It is well-known that the characteristic connection $\nabla^c$ can admit
a parallel spinor field in more situations than for the Levi-Civita connection 
$\nabla^g$, and that an analogue of Wang's classification result is not
possible. For example, any $G_2$ structure and any $\Spin(7)$ structure
admitting a characteristic connection $\nabla^c$ has a $\nabla^c$-parallel
spinor field, just because  $G_2$ and $\Spin(7)$ are the stabilizers of a
generic spinor in dimension $7$ and $8$, respectively (\cite{FI02}, \cite{I04}).
More recently, the twistor and Killing equations for the characteristic
connection were investigated in \cite{ABK12} and \cite{BB12}; we will speak of
\emph{Killing spinors with torsion} to distinguish them from the Riemannian
case.  Again, the picture is roughly as follows: there are more $G$ manifolds 
admitting Killing spinors with torsion  than in the Riemannian case, and their
geometry is less rigid (for example, they do not have to be Einstein, and
the Killing number is not automatically linked to the first eigenvalue of
the characteristic Dirac operator). This richness in turn implies that
a classification is not possible. One further crucial difference to the
Riemannian case is that the families of manifolds admitting parallel spinors
resp.~Killing spinors with torsion are not disjoint any more, both are
described in the language of $G$ structures sketched above and it is to be
discussed in every situation anew what can be said about particular spinor fields.
\subsection{Outline}
The main purpose of the present paper is to investigate the cone correspondence
for $G$ manifolds admitting a characteristic connection. While doing so,
several results are obtained that should be of interest in other
circumstances as well.

Section $2$ is devoted to the general construction. Given a Riemannian
manifold $(M,g)$ we denote by $(\bar{M},\bar{g})=(M\times\R^+,\,a^2r^{2}g+dr^2)$ for 
some fixed $a>0$ its cone (we sometimes call $a$ the \emph{cone constant} of
$\bar M$). Of course, the cone does always exist and carries
interesting geometric structures, but if one intends to lift a Killing spinor
with torsion from $M$ to $\bar M$, one has to choose $a$ suitably, depending
on the Killing number $\alpha$. It is crucial that $\alpha$ is not allowed
to vanish, i.\,e.~there is no cone correspondence for parallel spinor fields
(but see Corollary \ref{cor:g2explkilling} for an exception). 
The details of this `abstract' cone correspondence for most
general metric connections with skew torsion are explained  in Section 2.1.
Section 2.2 introduces the $G$ structures that will be of particular
interest in this article and their characteristic connections.
For  metric almost contact structures, we prove a new criterion for the
existence of a characteristic connection (Lemma
\ref{lem.contact-char-connection}) and describe the corresponding
Chinea-Gonzales  classes. For almost hermitian
structures,  $G_2$ structures, and $\Spin(7)$ structures, we quickly recall
about their characteristic connections a few facts that  we shall need later. 

In Section 2.3 we begin to sketch the details of the cone correspondence.
Suppose $M$ carries a $G$ structure with characteristic connection $\nabla^c$,
and that  we can define  a $\bar G$ structure on the cone $(\bar M, \bar
g)$ with characteristic connection $\bar\nabla^c$. 
\emph{Then an important obervation is that the lift of $\nabla^c$ to $\bar M$ is
not the characteristic connection $\bar\nabla^c$ of the  $\bar G$ 
structure on $\bar M$!} This happens already in the classical case covered
by B\"ar, where the characteristic connection on $M$ is not $\nabla^g$, while
the $\bar G$ structure on $\bar M$ is integrable, hence its characteristic
connection is equal to the Levi-Civita connection.
Rather, we  need as an intermediate step another  
connection $\nabla$ on $M$ with torsion $T$ such that its lift 
$\bar\nabla$ to $\bar M$ with 
torsion $\bar T$ is the characteristic connection on $\bar M$ with respect to 
the given $\bar G$ structure. The torsion $T$ measures in some sense the
deviation of the $\bar G$ structure from the integrable case, i.\,e.~the classical
cone correspondence describes the situations where $T=0$, hence $\bar T=0$ and
$\bar\nabla=\nabla^{\bar g}$. Lemma \ref{lem:corrksps} describes the
exact correspondence between Killing spinors with torsion on $M$
and $\bar\nabla$-parallel spinors on $\bar M$.

We  then describe in detail the cone correspondence with torsion for
two particular situations where $M$ is odd-dimensional.
Section 3 treats the case when $M$ is a metric almost contact manifold.
We construct an almost hermitian structure on $\bar M$, describe explicitly
the intermediate connection $\nabla$ and prove that its lift is the
characteristic connection of the almost hermitian structure. We then establish
the correspondence between the different classes of structures on $M$ and
$\bar M$, first through equations (Theorem \ref{thm:ocsclass}) and then in
terms of the different classes (\ref{thm:acsvsacs}). These results
synthesize several approaches to the definition of (some) metric almost
contact structures through the almost hermitian structures that they induce
on the cone (\cite{O62}, \cite{O85}); for normal structures ($N=0$), 
the correspondence was proved
independently in the recent preprints \cite{HTY12} and \cite{Conti&M12}
(see Remark \ref{NB.others} for details). 
In Section 3.3, the 
spinor correspondence is described in detail. In \cite{BB12}, it was
proved that the  Tanno deformation of a $(2n+1)$-dimensional 
Einstein-Sasaki manifold and that the $5$-dimensional Heisenberg group carry Killing
spinors with torsion. As an application, we prove in Section 3.4 
that these spinors lift
to spinors on the cone (it turns out to be conformally K\"ahler) 
that are parallel with respect to its characteristic  connection. 
Section 3.5 specializes the previous results to metric almost contact $3$-structures.

Section $4$ is devoted to the case when $M$ is a $G_2$ manifold. We construct
a $\Spin(7)$ structure on its cone, describe explicitly
the intermediate connection $\nabla$ and prove again that its lift is the
characteristic connection of the $\Spin(7)$ structure. In 4.2 we establish
the explicit correspondence between the different classes on $M$ and $\bar M$
(Lemma \ref{lem:spin7str} and Theorem \ref{thm.whenU1}); 
the results are slightly simpler than in the contact case, 
because the number of classes is smaller.
In 4.3 we establish again the details of the spinor correspondence. 
In Corollary \ref{cor:g2explkilling}, we prove by a clever interpretation of
the involved equations that the $\nabla^c$-parallel spinor defining the $G_2$
structure on $M$ lifts to a parallel spinor for the characteristic
connection of the $\Spin(7)$ structure on $\bar M$ -- thus, the spinor
correspondence turns out to be as neat as one could expect, and the use
of the intermediate connection $\nabla$ is not a draw back at all of the construction.

We end this outline with some words about the cone for even-dimensional
manifolds $M$. The most interesting case would be the lift from an
almost hermitian structure on $M$ to a $G_2$ structure on $\bar M$. As described
in several recent publications (\cite{Hitchin00}, \cite{Chiossi&S02},
\cite{Stock09}), 
the construction of a $G_2$ structure requires the use of Hitchin's flow
methods, and it is not very transparent how this could be generalized to cones
without having to solve a differential equation in the process. Thus, we
reserve such thoughts to a separate, upcoming publication. 
\subsection{Acknowledgements}
Both authors thank Thomas Friedrich (Berlin) for his steady
mathematical interaction.
Ilka Agricola acknowledges financial support by the
DFG within the priority programme 1388 "Representation theory".
Jos H\"oll thanks Philipps-Universit\"at Marburg for a Ph.\,D.~grant.
He is member of the `Graduate Center for Life and Natural Sciences'
of Philipps-Universit\"at Marburg and of the special graduate programme
`Lie theory and complex geometry' of the Department of Mathematics and
Computer Science.

%
\section{The general construction}
%
\subsection{The cone construction}\label{ch:constrcone}
%
Consider a Riemannian spin manifold $(M,g)$ equipped with a metric connection 
$\nabla$ with skew symmetric torsion $T$ and connection form $\omega$. We are 
interested in real Killing spinors with respect to the given connection, 
$\nabla_X\psi=\alpha X\psi$ with $\alpha\in\R\backslash\{0\}$. 
The aim of this Section is to generalize B\"ar's cone construction  \cite{B93} 
for Riemannian Killing spinors, i.\,e.~the case when $\nabla=\nabla^g$.
As an intermediate tool,
we define a connection $\tilde{\nabla}$ on the spinor bundle by 
\bdm
\tilde{\nabla}_X\psi=\nabla_X\psi+ \alpha X\cdot\psi, \mbox{ with } 
\alpha\in\R\backslash\{0\}. 
\edm
Denote by $\Cl(\R^n)$ the Clifford algebra of $\R^n$ with respect to the
standard negative definite euclidian scalar product, and by $\Delta_n$ the
spin module of $\Spin(n)$. We consider the Clifford multiplication for 
$X\in \R^n\subset \Cl(\R^n)$ in $\Delta_n$. It is the action of an element of 
$\R^n\subset\spin(n)\oplus\R^n=\spin(n+1) \subset \Cl(\R^n)$ in $\Delta_n$.
Let $P_{\SO(n)}M$ be the $\SO(n)$-principal bundle of frames,
$\Sigma M$  the spinor bundle  and $\rho_n:\Cl(n)\ra
GL(\Delta_n)$  the representation of the Clifford algebra, 
i.\,e.~$\rho_{*|\spin(n)}$ is the $\spin(n)$ representation. 
 Let  $P_{\Spin(n)}M$ be the $\Spin(n)$-principal bundle.
For a local section $h$ in $P_{\SO(n)}M$, we identify $TM$ and
$P_{\SO(n)}M\times_{SO(n)}\R^n$ via $X=[h,\eta(dh(X))]$, where $\eta$ is
the solder form. The affine connection $\tilde{\nabla}$ induces
 a connection in the $\Spin(n+1)$-principal bundle  
$P_{\Spin(n)}M\times_{\Spin(n)}\Spin(n+1)$ as follows. Let 
\bdm
\Phi:\ P_{\Spin(n)}M\ra P_{\SO(n)},\quad \theta:\ \Spin(n)\ra \SO(n)
\edm
be the usual 
projections. We look at $\spin(n+1)\cong\spin(n)\oplus\R^n\subset \Cl(n)$, the
restriction of $\rho_{*}$ to $\spin(n+1)$, and obtain for a local section $k$ in
$P_{\Spin(n)}M$ with $\Phi(k)=h$ and $\Sigma M\ni\psi=[k,\sigma]$, 
\begin{align*}
\tilde{\nabla}_X[k,\sigma]&=\nabla_X[k,\sigma]+\alpha\cdot[h,\eta(dhX)]
\cdot[k,\sigma]\\
&=[k,d\sigma(X)+\rho_*(\theta^{-1}_*(\omega(dhX))+\alpha\eta(dhX)))\sigma].
\end{align*}
Thus we get the $\spin(n+1)$-valued $1$-form 
$\hat{\omega}:=\Phi^*(\theta^{-1}_*\omega+\alpha\eta)$ on $P_{\Spin(n)}M$. We 
extend $\hat{\omega}$ to $P_{\Spin(n+1)}M$ as follows:
For $b\in P_{\Spin(n)}M$ we have $T_bP_{\Spin(n+1)}M=T_bP_{\Spin(n)}M\oplus
dL_b(\R^n)$, where $L_b:\Spin(n+1)\ra P_{\Spin(n+1)}M,~g\mapsto b\cdot g$ and define
\bdm
\hat{\omega}(dL_bY):=Y\in\R^n\subset\spin(n+1).
\edm
For any $b\in P_{\Spin(n)}M$ we further extend $\hat{\omega}$ in a 
$\Spin(n+1)$ equivariant way.
One checks that the given form is a connection form. It is the connection form 
of the connection given by $\tilde{\nabla}$.
As in \cite{B93}, we consider the $\SO(n+1)$-principal bundle 
\bdm
P_{\SO(n+1)}M\ :=\ P_{\SO(n)}M\times_{\SO(n)}\SO(n+1)
\edm
and calculate the corresponding 
connection form $\tilde{\omega}$ given by $\theta^{-1}_*\Phi^*\tilde{\omega}
=\hat{\omega}$ for the projections $\Phi:P_{\Spin(n+1)}M\ra P_{\SO(n+1)}M$ 
and $\theta:\Spin(n+1)\ra \SO(n+1)$ and get
\bdm
\tilde{\omega}=\begin{bmatrix}\omega & -2\alpha\eta \\ 
2\alpha\eta^t & 0 \end{bmatrix}.
\edm
%
%
%
%
%
We now consider the cone  $(\bar{M},\bar{g})=(M\times\R^+,a^2r^{2}g+dr^2)$ for 
some fixed $a>0$ with principal $\SO(n)$-bundle of frames
$P_{\SO(n+1)}\bar{M}$,  Levi-Civita connection  $\bar{\nabla}^{\bar{g}}$
with  connection form $\bar{\omega}^{\bar{g}}$ and projection 
$\pi:\bar{M}\ra M$. For simplicity, we will write $X\in TM$ for a lift to 
$\bar M$ of a vector field on $M$.
We define a tensor $\bar{T}$ on $\bar{M}$ from the torsion tensor $T$ of
$\nabla$ via 
\bdm
\bar{T}(X,Y)\ :=\ T(X,Y) \ \text{ for }X,Y\perp \del_r,\quad
\del_r\lrcorner \bar{T}\ =\ 0. 
\edm
Looking at the corresponding skew symmetric 
$3$-tensors and the metrics $g,\bar g$ on $M$ and $\bar M$, we have 
$a^2r^2T(X,Y,Z)=\bar T(X,Y,Z)$ for $X,Y,Z\perp \del_r$. From $\bar T$, we define
on $\bar M$ the connection 
\bdm
\bar{\nabla}\ :=\ \nabla^{\bar{g}}+\frac{1}{2}\bar{T},
\edm
whose connection form is $\bar{\omega}$.
For $p\in M$ and $s\in \R^+$, the tangent bundle of $\bar{M}$ splits into 
$T_{(p,s)}\bar{M}=T_pM\oplus \R$, where $d\pi(T\bar{M})=TM$. Thus, for 
$X\in TM\subset T\bar{M}$, we will write $"X"$ instead of $"d\pi X"$.
With a local orthonormal frame $(X_1,\ldots, X_n)$ of $M$ we have an 
isomorphism
of the last two vector bundles given by ($Y\in\R^{n+1}$)
\bdm
\phi:\  \pi^*(\tilde{P}_{\SO(n+1)}M)\times_{\SO(n+1)}\R^{n+1} 
\ra T\bar{M}, \ ~ [(X_1, .. ,X_n,\del_r),Y] \mapsto 
[(\frac{1}{ar}X_1, .. ,\frac{1}{ar}X_n,\del_r),Y].
\edm
Thus we can view the connection $\bar{\omega}$ as a 
connection of $\pi^*(\tilde{P}_{\SO(n+1)}M)$, which we again call
$\bar{\omega}$.

We summarize the different principal bundles with corresponding connections
and vector bundles  in the following table:
\begin{center}
\begin{tabular}{c|c|c|c}
bundle&connection form & vector bundle & manifold\\
\hline
$P_{\SO(n)}M$ & $\omega$ & $TM$ & $M$\\
$\tilde{P}_{\SO(n+1)}M$ & $\tilde{\omega}$ &&$M$\\
$\pi^*(\tilde{P}_{\SO(n+1)}M)$ &$\pi^*\tilde\omega$ 
&$\pi^*(\tilde{P}_{\SO(n+1)}M)\times_{\SO(n+1)}\R^{n+1}$ &$\bar M$\\
$P_{\SO(n+1)}\bar{M}$ & $\bar{\omega}$ & $T\bar{M}$ & $\bar M$
\end{tabular}
\end{center}

\smallskip

 To determine $\bar{\omega}$ for a local frame 
$h:=(X_1, .. ,X_n,\del_r)$ in $\pi^*(\tilde{P}_{\SO(n+1)}M)$, $X\in T\bar{M}$, 
we need to compute
($Y\in \pi^*(\tilde{P}_{\SO(n+1)}M)\times_{\SO(n+1)}\R^{n+1}$)
\bdm
\phi^{-1}(\bar{\nabla}_X\phi(Y))\ =\ [h,d(\eta(dhY))(X)+\bar{\omega}(dhX)\eta(dhY)].
\edm
Let $\tilde{h}:=(\frac{1}{ar}X_1, .. ,\frac{1}{ar}X_n,\del_r)$ be a local 
frame in $P_{\SO(n+1)}$. For $Y\in TM\subset \pi^*(\tilde{P}_{\SO(n+1)}M)
\times_{\SO(n+1)}\R^{n+1}$ we locally have $Y=[h,(Y_1,..,Y_n,0)^t]$ for 
functions $Y_i:M\ra \R$ and thus $\phi(Y)=[\tilde{h},(Y_1,..,Y_n,0)^t]$. 
Therefore $ar\phi(Y)$ is independent of $r$ and thus a lift of a vector 
field on $M$.
Using the  O'Neill formulas \cite[p. 206]{O83}, we compute for lifts $X,Y$ of 
vector fields in $TM$ and the Levi-Civita connection 
$\bar{\nabla}^{\bar{g}}$ of $\bar{M}$
\bdm
\bar{\nabla}^{\bar{g}}_{\del_r}\del_r=0, \quad
\bar{\nabla}^{\bar{g}}_{\del_r}X=\bar{\nabla}^{\bar{g}}_X\del_r=\frac{1}{r}X,
\quad
\bar{\nabla}^{\bar{g}}_XY=\nabla^g_XY-\frac{1}{r}\bar{g}(X,Y)\del_r.
\edm
Adding the torsion tensor $\bar{T}$, this implies
\bdm
\bar{\nabla}_{\del_r}\del_r=0,\quad
\bar{\nabla}_{\del_r}X=\bar{\nabla}_X\del_r=\frac{1}{r}X,\quad
\bar{\nabla}_XY=\nabla_XY-\frac{1}{r}\bar{g}(X,Y)\del_r.
\edm
For $X\in T\bar{M}$ and $Y\in TM\subset\pi^*(\tilde{P}_{\SO(n+1)}M)
\times_{\SO(n+1)}\R^{n+1}$ we have
\bdm
\phi^{-1}(\bar{\nabla}_{\del_r}\phi(\del_r))=\phi^{-1}(\bar{\nabla}_{\del_r}\del_r)=0\stackrel{!}{=}[h,d((0..0,1)^t)(\del_r)+\bar{\omega}(dh\del_r)(0..0,1)^t]=[h,\bar{\omega}(dh\del_r)(0..0,1)^t]
\edm
and
\begin{align*}
\phi^{-1}(\bar{\nabla}_{\del_r}\phi(Y))&
=\phi^{-1}(\bar{\nabla}_{\del_r}\frac{1}{ar}ar\phi(Y))
=\phi^{-1}(\frac{1}{ar}\bar{\nabla}_{\del_r}ar\phi(Y)+(\del_r\frac{1}{ar})ar\phi(Y))\\
&=\phi^{-1}(\frac{1}{ar}\frac{1}{r}(ar\phi(Y))-\frac{1}{ar^2}ar\phi(Y))=0\\
&\stackrel{!}{=}[h,0+\bar{\omega}(dh\del_r)(Y_1,..,Y_n,0)^t]
\end{align*}
and thus $\bar{\omega}(dh\del_r)=0$. 
Furthermore $X=[\tilde{h},ar(X_1,..,X_n,0)^t]=[\tilde{h},ar\eta(dhX)]$ and we get
\bdm
\phi^{-1}(\bar{\nabla}_{X}\phi(\del_r))
=\phi^{-1}(\bar{\nabla}_{X}\del_r)
=\phi^{-1}(\frac{1}{r}X)=\phi^{-1}([\tilde{h},a\eta(dhX)])
=[h,a\eta(dhX)],
\edm
proving $a\eta=\bar\omega\cdot\del_r$.
Since $\phi(Y)=[\tilde{h},(Y_1,..,Y_n,0)^t]$, we have 
$\bar{g}(X,ar\phi(Y))=a^2r^2\eta(dhX)^t\cdot (Y_1,..,Y_n)^t$. Furthermore we have
\bdm
\nabla_Xar\phi(Y)=[\tilde{h},ar(d(Y_1,..Y_n,0)^t(X)
+ar(\omega(dhX)(Y_1,..,Y_n)^t,0)^t]
\edm
and obtain
\begin{align*}
 &\phi^{-1}(\bar{\nabla}_{X}\phi(Y))
=\phi^{-1}(\frac{1}{ar}\bar{\nabla}_Xar\phi(Y))
=\phi^{-1}(\frac{1}{ar}\nabla_Xar\phi(Y)
-\frac{1}{ar}\frac{1}{r}\bar{g}(X,ar\phi(Y))\del_r)\\
&=\phi^{-1}([\tilde{h},d(Y_1,..,Y_n,0)^t(X)
+(\omega(dhX)(Y_1,..,Y_n)^t,0)^t-a\eta(dhX)^t(Y_1,..,Y_n,0)^t(0,..,0,1)^t]).
\end{align*}
Combining all these results yields
\bdm
\bar{\omega}=\begin{bmatrix}\omega & a\eta \\ -a\eta^t & 0 \end{bmatrix}.
\edm
If one changes the orientation of $\bar{M}$ (a local $\SO(\bar{M})$ frame is
then given by $(\frac{1}{ar}X_1,...,\frac{1}{ar}X_2,-\del_r$)), we obtain
the alternative connection form
\bdm
\begin{bmatrix}\omega & -a\eta \\ a\eta^t & 0 \end{bmatrix}.
\edm
For  a Killing spinor  on $M$ with real Killing number $\alpha$, we thus choose
the cone constant $a=-2\alpha$ for $\alpha<0$ and $a=2\alpha$ for $\alpha>0$. Hence,
the cone \emph{depends} on the Killing number and the construction only makes 
sense if $\alpha\in\R\backslash\{0\}$, as we had assumed from the beginning.
In particular, the results cannot be applied to $\nabla$-parallel spinors
($\alpha=0$).  The pullback 
of the connection $\tilde{\omega}$ under the projection $\pi:\bar{M}\ra M$ 
is the same as the connection $\bar{\omega}$ on $\bar{M}$, thus 
their holonomy groups $Hol(\tilde{\omega})$ and $Hol(\bar{\omega})$ are the
same. Since the second Stiefel-Whitney class of $\bar M=M\times\R$ is given by 
\cite[p.142]{T52}
\bdm
\mathit{w}_2(\bar M)\ =\ \mathit{w}_2(M)+\mathit{w}_2(\R)
+\mathit{w}_1(M)\otimes\mathit{w}_1(\R),
\edm
we conclude that $\bar M$ is spin, since we assumed $M$ to be spin.

Let us now have a closer look at spinors on $M$ and $\bar M$.
A parallel spinor of $(\bar{M},\bar{\omega})$ is the same as a trivial factor 
of the action of the holonomy group $Hol(\bar{\omega})=Hol(\tilde{\omega})$ 
on $\Delta_{n+1}$. A Killing spinor on ($M$,$\omega$) corresponds to a 
trivial factor of the action of the same group on the space $\Delta_n$.

For $n=\dim(M)$ odd,  the spin representation splits into 
$\Delta_{n+1}=\Delta_n^+\oplus\Delta_n^-$. Changing the orientation of
$\bar{M}$ (changing from negative to positive $\alpha$ and vice versa) means 
interchanging $\Delta_n^+$ and $\Delta_n^-$. Thus, a parallel spinor on
$\bar{M}$ is either in $\Delta_n^+$ or in $\Delta_n^-$, giving either a
Killing spinor with positive or with negative Killing number $\alpha$. 

For $n$ even, we have $\Delta_n=\Delta_{n+1}$ and, by interchanging the
orientation, we obtain for any parallel spinor in $\bar{M}$ one Killing 
spinor with positive, and one with negative Killing number $\alpha$. 
We summarize these results in the following lemma:
\begin{lem}\label{lem:ks}
 For a Riemannian spin manifold $(M,g)$ with connection $\nabla$ with 
skew symmetric torsion $T$, consider the manifold $(\bar M,\bar g)$ with 
connection $\bar \nabla$ with skew symmetric torsion $\bar T$ as constructed
above. The following  correspondence holds:
\begin{itemize}
 \item If $n=\dim(M)$ is odd, any $\bar \nabla$-parallel spinor on $\bar M$ 
corresponds to a $\nabla$-Killing spinor on $M$, with either positive or 
negative Killing number $\frac{1}{2}a$ or $-\frac{1}{2}a$.
\item If $n$ is even, any $\bar \nabla$-parallel spinor on $\bar M$
  corresponds to a pair of $\nabla$-Killing spinors on $M$ with Killing 
number $\pm\frac{1}{2}a$.
\end{itemize}
\end{lem}
\begin{NB}
For $\dim M$ even, one can write down the bijection between Killing spinors
with torsion with Killing numbers $\pm \alpha$ explicitly: If $\psi$
has Killing number $\alpha$ and decomposes into $\psi=\psi_+ +\psi_-$ in the
spin bundle $\Sigma M=\Sigma^+ M\oplus \Sigma^- M$,
then $\psi_+ - \psi_-$ is a  Killing spinor with 
Killing number $-\alpha$. This is the same argument as in the Riemannian case
\cite[p.121]{BFGK}.
\end{NB}
\begin{NB}
The careful reader will have noticed that our cone is slightly more
general than in \cite{B93}, where the computations are done for cone constant
$a=1$. This stems from
the fact that in the Riemannian case, the Killing number is determined through
$n=\dim M$ and $\Scal^g$ (remember that the manifold has to be Einstein), 
hence the cone can be normalized in such a way
that $a=1$. For our applications, this is too restrictive. 

\end{NB}
%
%
\subsection{$G$ structures and their characteristic connections}%
\label{sec:som}
%
%
Let $(M,g)$ be an oriented Riemannian manifold with Levi-Civita 
connection $\nabla^g$. 
By definition, a  $G$ structure on $M$ is a $G$ reduction of the frame bundle
of $M$ to some closed  subgroup $G\subset\SO(n)$. If $M$ admits a 
metric connection $\nabla^c$ with skew symmetric torsion $T^c$
preserving the $G$ structure, it will be called a \emph{characteristic connection}.
The following result proves the uniqueness of the characteristic connection
in many geometric situations:
\begin{thm}[{\cite[Thm 2.1.]{AFH12}}]\label{char-conn-unique}
Let $G\subsetneq\SO(n)$ be a connected Lie subgroup acting irreducibly on 
$\R^n$, and assume that $G$ does not act on $\R^n$ by its adjoint 
representation. Then the characteristic connection of a 
$G$ structure on a Riemannian manifold $(M,g)$ is, if existent, unique.
\end{thm}
This applies, for example, to almost hermitian structures
($\SU(n)\subset\SO(2n)$), $G_2$ structures in dimension $7$ and $\Spin(7)$ 
structures in dimension $8$ (but not to metric almost contact structures).

Let us introduce the $G$ structures considered in this article.
%
\subsubsection*{\underline{Metric almost contact structures}}
%
Let $M$ be a $n=2k+1$ dimensional manifold. Given a Riemannian
metric $g$, a (1,1)-tensor $\phi:TM\ra TM$, a $1$-form $\eta$ with dual vector field $\xi$ of length one, and 
the  $(2,0)$-tensor $F$ defined by $F(v,w):=g(v,\phi(w))$,
we call $(M,g,\phi,\eta)$ a metric almost contact structure if 
\bdm
\phi^2=-id+\eta\otimes\xi~~~~~ \mbox{ and } ~~~~~g(\phi v,\phi w)
=g(v,w)-\eta(v)\eta(w).
\edm
In \cite[Thm 4.1.D]{B02}, D.~Blair shows that $\phi(\xi)=0$ and $\eta\circ\phi=0$.
Since 
\bdm
g(v,\phi(w))\ =\ g(\phi(v),\phi^2(w))+\eta(v)\eta(\phi(w))
\ =\ g(\phi(v),-w+\eta(w)\xi)\ =\ -g(\phi(v),w),
\edm
for all $v,w\in TM$, $F$ is actually a $2$-form.
In terms of the Levi-Civita connection $\nabla^g$ on $M$, the Nijenhuis tensor of a metric almost contact structure is defined by
\begin{align*}
N(X,Y,Z)\ :=\ &g((\nabla^g_X\phi)(\phi(Y))-(\nabla^g_Y\phi)(\phi(X))+(\nabla^g_{\phi(X)}\phi)(Y)-(\nabla^g_{\phi(Y)}\phi)(X),Z)\\
&+\eta(X)g(\nabla^g_Y\xi,Z)-\eta(Y)g(\nabla^g_X\xi,Z).
\end{align*}
The classification of metric almost contact  structures is, alas, relatively
involved. For future reference, we recall in the following table
the exact definition of the different
classes of of $n$-dimensional metric almost contact manifolds given by 
Chinea and Gonzalez \cite{CG90}.

\smallskip
\begin{center}
\begin{tabular}{|c|c|c|}
 \hline
 class & defining relation\\
\hline
\hline
$\mathcal{C}_1$ & $(\nabla^g_XF)(Y,Z)=0$, $\nabla^g\eta=0$\\
\hline
$\mathcal{C}_2$ & $dF=\nabla^g\eta=0$\\
\hline
$\mathcal{C}_3$ & $(\nabla^g_XF)(Y,Z)-(\nabla^g_{\phi X}F)(\phi Y,Z)=0$\\
\hline
$\mathcal{C}_4$ & $(\nabla^g_XF)(Y,Z)=-\frac{1}{n-3}[g(\phi X,\phi Y)\delta F(Z)-g(\phi X,\phi Z)\delta F(Y)$\\&$-F(X,Y)\delta F(\phi Z)+F(X,Z,\delta F(\phi Y)]$, ~~~$\delta F(\xi)=0$\\
\hline
$\mathcal{C}_5$ & $(\nabla^g_XF)(Y,Z)=\frac{1}{n-1}[F(X,Z)\eta(Y)-F(X,Y)\eta(Z)]\delta\eta$\\
\hline
$\mathcal{C}_6$ & $(\nabla^g_XF)(Y,Z)=\frac{1}{n-1}[g(X,Z)\eta(Y)-g(X,Y)\eta(Z)]\delta F(\xi)$\\
\hline
$\mathcal{C}_7$ & $(\nabla^g_XF)(Y,Z)=\eta(Z)(\nabla^g_Y\eta)(\phi X)+\eta(Y)(\nabla^g_{\phi X}\eta)(Z)$,~~~ $\delta F=0$\\
\hline
$\mathcal{C}_8$ & $(\nabla^g_XF)(Y,Z)=-\eta(Z)(\nabla^g_Y\eta)(\phi X)+\eta(Y)(\nabla^g_{\phi X}\eta)(Z)$,~~~ $\delta \eta=0$\\
\hline
$\mathcal{C}_9$ & $(\nabla^g_XF)(Y,Z)=\eta(Z)(\nabla^g_Y\eta)(\phi X)-\eta(Y)(\nabla^g_{\phi X}\eta)(Z)$\\
\hline
$\mathcal{C}_{10}$ & $(\nabla^g_XF)(Y,Z)=-\eta(Z)(\nabla^g_Y\eta)(\phi X)-\eta(Y)(\nabla^g_{\phi X}\eta)(Z)$\\
\hline
$\mathcal{C}_{11}$ & $(\nabla^g_XF)(Y,Z)=-\eta(X)(\nabla^g_\xi F)(\phi Y,\phi Z)$\\
\hline
$\mathcal{C}_{12}$ & $(\nabla^g_XF)(Y,Z)=\eta(X)\eta(Z)(\nabla^g_\xi\eta)(\phi Y)-\eta(X)\eta(Y)(\nabla^g_\xi\eta)(\phi Z)$\\
 \hline
\end{tabular}
\end{center}

\medskip
The most important classes are 
\begin{itemize}
 \item $\mathcal C_3\oplus..\oplus\mathcal C_8$, the normal structures 
characterized by $N=0$,
\item $\mathcal C_6\oplus C_7$, the quasi Sasaki structures: normal 
structures satisfying $dF=0$,
 \item $\mathcal C_6$, the $\alpha$-Sasaki structures:
normal structures with $\alpha F=d\eta$ for some constant $\alpha$,
 \item Sasaki structures: $\alpha$-Sasaki structures with $\delta F(\xi)=n-1$.
\end{itemize}
Other classifications we will not consider here are formulated in terms of 
the Niejenhuis tensor or by considering the direct (not the twisted) 
product $M\times \R$ (\cite{CM92} and \cite{O85}).
It turns out that the tensor
$\alpha(X,Y,Z):=(\nabla^g_{X}F)(Y,Z)$ will be a useful tool for the
investigation of metric almost contact structures. It satisfies 
the general formula
\be\label{eq:genforacs}
\alpha(X,Y,Z)\,=\, -\alpha(X,Z,Y)\, =\, 
-\alpha(X,\phi Y,\phi Z)+\eta(Y)\alpha(X,\xi,Z)+\eta(Z)\alpha(X,Y,\xi).
\ee
This implies 
\bdm
\alpha(X,Y,\phi Y)\, =\,
-\alpha(X,\phi Y,\phi^2 Y)+\eta(Y)\alpha(X,\xi,\phi Y)\,
=\, -\alpha(X,Y,\phi Y)+2\eta(Y)\alpha(X,\xi,\phi Y),
\edm
so we have
\begin{equation}\label{eq:acsalphaphi}
\alpha(X,Y,\phi Y)=\eta(Y)\alpha(X,\xi,\phi Y).
\end{equation}
A metric almost contact structure admits a  characteristic connection 
if and only if its Nijenhuis tensor is skew symmetric and $\xi$ is a Killing 
vector field, and it is then unique \cite[Thm 8.2]{FI02}. If it exists, its torsion tensor is given by
\bdm
T\ =\ \eta\wedge d\eta +dF^\phi+N-\eta\wedge(\xi\lrcorner N),
\edm
where $dF^\phi:=dF\circ\phi$. 
We shall now prove a useful criterion for the existence of a characteristic 
connection.
\begin{lem}\label{lem.contact-char-connection}
A metric almost contact manifold $(M,g,\phi,\eta)$ admits a 
characteristic connection if and only if
\bdm
(\nabla^g_YF)(Y,\phi X)+(\nabla^g_{\phi Y}F)(Y,X)\ =\ 0.
\edm
\end{lem}
\begin{proof}
There exists a characteristic connection if and only if the Niejenhuis 
tensor $N$ is skew symmetric and $\xi$ is a Killing vector field. 
Since we have 
\bdm
g(\nabla^g_Y\xi,Z)\, =\, -F(\nabla^g_Y\xi,\phi Z)\,
= \, (\nabla^g_YF)(\xi,\phi Z)\, =\, (\nabla^g_Y\eta)(Z)
\edm
and $(\nabla^g_XF)(Z,Y)=g((\nabla^g_X\phi)Y,Z)$, the Niejenhuis tensor on $M$ 
may be written as 
\begin{align*} 
N(X,Y,Z)~= ~&\alpha(X,Z,\phi Y)-\alpha(Y,Z,\phi X)+\alpha(\phi X,Z,Y)
-\alpha(\phi Y,Z,X)\\
&+\eta(X)\alpha(Y,\xi,\phi Z)-\eta(Y)\alpha(X,\xi,\phi Z).
\end{align*}
Thus $N$ is skew symmetric if 
\bdm
0=N(X,Y,Y)= \alpha(X,Y,\phi Y)-\alpha(Y,Y,\phi X)-\alpha(\phi Y,Y,X)
+\eta(X)\alpha(Y,\xi,\phi Y)-\eta(Y)\alpha(X,\xi,\phi Y).
\edm
%
With equation (\ref{eq:acsalphaphi}),  $N$ is skew symmetric if and only if
\begin{equation}\label{eq:Nschief}
0\ =\ \alpha(Y,Y,\phi X)+\alpha(\phi Y,Y,X)+\eta(X)\alpha(Y,\xi,\phi Y).
\end{equation}
$\xi$ is a Killing vector field if 
$0=g(\nabla^g_X\xi,Y)+g(\nabla^g_Y\xi,X)=\alpha(X,\xi,\phi Y)
+\alpha(Y,\xi,\phi X)$, and this is satisfied if and only if 
$\alpha(Y,\xi,\phi Y)=0$. Together with condition (\ref{eq:Nschief}) 
we obtain the condition
\bdm
0\ =\ \alpha(Y,Y,\phi X)+\alpha(\phi Y,Y,X).
\edm
To see that this is also sufficient, set $X=\xi$.
\end{proof}
\begin{dfn}
In analogy to the almost hermitian and the $G_2$ case, we shall
call a metric almost contact manifold admitting a characteristic connection 
a \emph{metric almost contact manifold with torsion}.
\end{dfn}
With the above lemma we can easily prove
\begin{thm}\label{thm:acscharconn}
Consider a metric almost contact manifold $(M,g,\phi,\eta)$. 
If it is of class
\begin{enumerate}
 \item $\mathcal{C}_1\oplus\mathcal{C}_3\oplus\mathcal{C}_4
\oplus\mathcal{C}_{6}\oplus\mathcal{C}_{7}$, there exists a characteristic connection.
 \item $\mathcal{C}_2$, $\mathcal{C}_5$, $\mathcal{C}_9$, $\mathcal{C}_{10}$, 
$\mathcal{C}_{11}$ or $\mathcal{C}_{12}$ there is no characteristic connection.
 \item $\mathcal{C}_8$ there exists a characteristic connection if and only 
if $\xi$ is a Killing vector field. 
\end{enumerate}
\end{thm}
\begin{proof}
 We check the different cases:\\
In $\mathcal{C}_1$ we have $\alpha(X,X,Y)=\alpha(X,Z,\xi)=0$ and we thus 
get $\alpha(Y,Y,\phi X)+\alpha(\phi Y,Y,X)=0$.\\
For a structure given by $\alpha$ in the class $\mathcal{C}_2$ we have 
\bdm
\alpha(X,Y,Z)+\alpha(Y,Z,X)+\alpha(Z,X,Y)\ =\ \alpha(X,Y,\xi)\ =\ 0,
\edm
and equation (\ref{eq:acsalphaphi}) yields
\begin{align*}
\alpha(Y,Y,\phi X)+\alpha(\phi Y,Y,X)=\ &\alpha(Y,Y,\phi X)
-\alpha(Y,X,\phi Y)-\alpha(X,\phi Y,Y)\\
=\ &\alpha(Y,Y,\phi X)+\alpha(Y,\phi Y,X)\stackrel{(\ref{eq:genforacs})}{=}
-\alpha(Y,\phi Y,\phi^2 X)+\alpha(Y,\phi Y,X)\\
=\ &2\alpha(Y,Y,\phi X).
\end{align*}
Thus the condition $\alpha(Y,Y,\phi X)+\alpha(\phi Y,Y,X)=0$ implies 
$0=\alpha(Y,Y,\phi^2 X)=-\alpha(Y,Y,X)$ since $\alpha(Y,Y,\xi)=0$. Therefore 
$\alpha$ has to be also of class $\mathcal{C}_1$, which implies $\alpha=0$.\\
In $\mathcal{C}_3$ we have $\alpha(X,Y,Z)=\alpha(\phi X,\phi Y,Z)$ and get
\begin{align*}
\alpha(Y,Y,\phi X)+\alpha(\phi Y,Y,X)\,=\, &\,
\alpha(Y,Y,\phi X)-\alpha(\phi Y,X,Y)\\
=& \,\alpha(Y,Y,\phi X)-\alpha(\phi^2 Y,\phi X,Y)=\alpha(Y,Y,\phi X)
+\alpha(Y,\phi X,Y)=0
\end{align*}
since $\alpha(\xi,X,Y)=0$ in $\mathcal{C}_1\oplus...\oplus\mathcal{C}_{10}$.\\
A structure is of class $\mathcal{C}_3\oplus...\oplus\mathcal{C}_{8}$ 
if and only if $N=0$ thus we just have to check the condition 
$\alpha(Y,\xi,\phi Y)=0$, which is satisfied in $\mathcal{C}_4$ 
and $\mathcal{C}_6$.\\
$\mathcal{C}_5$ is given by the condition 
$\alpha(X,Y,Z)=\frac{\delta \eta}{n-1}(F(X,Z)\eta(Y)-F(X,Y)\eta(Z))$ 
such that the condition $\alpha(Y,\xi,\phi Y)=0$ implies $\delta\eta=0$ 
and thus $\alpha=0$.\\
For $(c,b)=(1,-1)$ in $\mathcal{C}_7$, $(c,b)=(-1,-1)$ in $\mathcal{C}_8$, $(c,b)=(1,1)$ in $\mathcal{C}_9$ and $(c,b)=(-1,1)$ in $\mathcal{C}_{10}$ we have
\bdm
\alpha(X,Y,Z)\ =\ c\eta(Z)\alpha(Y,X,\xi)+b\eta(Y)\alpha(\phi X,\phi Z,\xi)
\edm
and get $\alpha(X,Y,\xi)=c\alpha(Y,X,\xi)$ and 
$\alpha(X,\phi Y,\xi)=b\alpha(X,\phi Y,\xi)$, implying 
$(1-cb)\alpha(Y,\phi Y,\xi)=0$. Thus in $\mathcal{C}_7$ and 
$\mathcal{C}_{10}$ the vector field $\xi$ is Killing. Since in 
$\mathcal{C}_7$ we have $N=0$, we have a characteristic connection here. 
In $\mathcal{C}_8$ we have a characteristic connection if and only if 
$\xi$ is Killing.
In $\mathcal{C}_9$ and $\mathcal{C}_{10}$ we have $b=1$ and thus
\begin{align*}
\alpha(Y,Y,\phi X)+\alpha(\phi Y,Y,X)=&-\eta(Y)
\alpha(\phi Y,X,\xi)+c\eta(X)\alpha(Y,\phi Y,\xi)-\eta(Y)\alpha(Y,\phi X,\xi)\\
=&-2\eta(Y)\alpha(\phi Y,X,\xi)+c\eta(X)\alpha(Y,\phi Y,\xi).
\end{align*}
For $X=\xi$ the condition $\alpha(Y,Y,\phi X)+\alpha(\phi Y,Y,X)=0$ 
implies $\alpha(Y,\phi Y,\xi)=0$ and thus we have $0=\alpha(\phi Y,X,\xi)$ 
and also $0=\alpha(\phi^2 Y,X,\xi)=-\alpha(Y,X,\xi)$ since $\alpha(\xi,X,Y)=0$. 
So we have already $\alpha=0$.\\
$\mathcal{C}_{11}$ is given by the condition 
$\alpha(X,Y,Z)=-\eta(X)\alpha(\xi,\phi Y,\phi Z)$ and thus with 
$\alpha(\xi,\xi,X)=0$ we get
\bdm
\alpha(Y,Y,\phi X)+\alpha(\phi Y,Y,X)\ =\ \eta(Y)\alpha(\xi,\phi Y,X).
\edm
Because $\alpha(\xi,\phi Y,X)=0$ already implies $\alpha(\xi,Y,X)=0$, 
we obtain in this case  immediately $\alpha=0$.\\
In $\mathcal{C}_{12}$ we have 
$\alpha(X,Y,Z)=\eta(X)\eta(Y)\alpha(\xi,\xi,Z)+\eta(X)\eta(Z)\alpha(\xi,Y,\xi)$ 
and thus $0=\alpha(Y,Y,\phi X)+\alpha(\phi Y,Y,X)
=\eta(Y)^2\alpha(\xi,\xi,\phi X)$ gives us $\alpha=0$.
\end{proof}
\begin{NB}
 The conditions for a metric almost contact structure to admit a characteristic 
connection in Theorem \ref{thm:acscharconn} are sufficient but not necessary. 
In \cite{P12} C. Puhle proves that in the case $n=5$, there are structures 
of class $\mathcal{C}_{10}\oplus\mathcal{C}_{11}$ (in his class
$\mathcal{W}_4$)
 carrying a characteristic connection. Thus a structure with characteristic connection 
is never of pure class $\mathcal{C}_{10}$ nor of class $\mathcal{C}_{11}$, but it 
can be of mixed class $\mathcal{C}_{10}\oplus\mathcal{C}_{11}$.
 But more detailed descriptions are possible in some cases. For example,
if we set  $Y=\xi$, the equation $0=\alpha(Y,Y,\phi X)+\alpha(\phi Y,Y,X)$ 
immediately implies that a structure with characteristic connection is of 
class $\mathcal{C}_1\oplus...\oplus\mathcal{C}_{11}$.
\end{NB}
%
\subsubsection*{\underline{Almost hermitian structures}}
%
Let $(M,g)$ be a $2m$-dimensional Riemannian manifold equipped with a $(1,1)$-tensor
\bdm
J:TM\ra TM ~~~\mbox{ with }~~~ J^2=-\Id_{TM}, ~~~\mbox{ and }~~~ g(JX,JY)=g(X,Y).
\edm
We define a $2$-form $\Omega(X,Y):=g(X,JY)$.
Then $(M,g,J,\Omega)$ is called an almost hermitian manifold. 
In terms of the Levi-Civita connection $\nabla^g$ on $M$, the Nijenhuis 
tensor of $M$ is defined to be
\begin{align*}
N(X,Y,Z)&=g((\nabla^{ g}_XJ)(JY),Z)-g((\nabla^{g}_YJ)(JX),Z)
+g((\nabla^{g}_{JX}J)(Y),Z)- g((\nabla^{g}_{JY}J)(X),Z).
\end{align*}
Almost hermitian structures were classified by Gray and Hervella in \cite{GH80} into four classes $\chi_1\oplus\chi_2\oplus\chi_3\oplus\chi_4$. An almost hermitian manifold admits a characteristic connection if and only if it is of class 
$\chi_1\oplus\chi_3\oplus\chi_4$ \cite{FI02} and it is always unique
(either by explicit computation as in \cite{FI02} or by the general Theorem 
\ref{char-conn-unique}); manifolds of class $\chi_1\oplus\chi_3\oplus\chi_4$
are sometimes called \emph{K\"ahler manifolds with torsion}, although they
are evidently not K\"ahlerian.   
Their characteristic torsion is given by (see for example \cite{A06})
\bdm
 T\ =\  N+d\Omega^J,
\edm
where $d\Omega^J:=d\Omega\circ J$. For a nearly K\"ahler manifold (class $\chi_1$),
this connection was first introduced and investigated by A.~Gray; on Hermitian
manifolds ($N=0$, i.\,e.~class $\chi_3\oplus\chi_4$) it is sometimes called the
\emph{Bismut connection} \cite{Bismut}. Almost hermitian manifolds of class $\chi_4$ are
locally conformally K\"ahler manifolds.

%
\subsubsection*{\underline{$G_2$ structures}}
%
Let $(M,g,)$ be a $7$-dimensional oriented Riemannian manifold. $M$ is said to
carry a $G_2$ structure if it admits a reduction to $G_2\subset \SO(7)$;
alternatively, this amounts to the choice of a generic $3$-form $\phi$.
With respect to a local orthonormal frame $e_1,\ldots,e_7$, such a $3$-form
 can locally be written as
\bdm
\phi=e_{123}+e_{145}+e_{167}+e_{246}-e_{147}-e_{347}-e_{356}.
\edm
Here and subsequently, we do not distinguish between vectors and covectors and 
abbreviate the $k$-form $e_{i_1}\wedge..\wedge e_{i_k}$ as $e_{i_1..i_k}$.
$G_2$ manifolds were classified by Fern\'andez and Gray in \cite{FG82} into four
classes $\mathcal W_1\oplus\mathcal W_2\oplus\mathcal W_3\oplus\mathcal W_4$.

Friedrich and Ivanov proved that there is a characteristic connection if and only 
if the structure is of class $\mathcal W_1\oplus\mathcal W_3\oplus\mathcal
W_4$; these manifolds are sometimes called \emph{$G_2$ manifolds with torsion} or
\emph{$G_2T$ manifolds} for short. 
In \cite{FI02} a concrete description of the torsion can be found (we do not
need the explicit formula here). We will often used the skew symmetric endomorphism 
$P(X,.)$ introduced in \cite{FG82},
\bdm
\phi(X,Y,Z)\, =\, g(X,P(Y,Z)).
\edm
%
\subsubsection*{\underline{$\Spin(7)$ structures}}
%
In a similar spirit, an  $8$-dimensional oriented Riemannian manifold $(M,g)$ is 
called a $\Spin(7)$ manifold if it has a reduction to $\Spin(7)\subset\SO(8)$,
and this is equivalent to the choice of a $4$-form $\Phi$ which,
in a local frame $e_1,\ldots,e_8$, can be written as
\bdm
\Phi=\phi+*\phi,\mbox{ and } 
\phi=e_{1278}+e_{3478}+e_{5678}+e_{2468}-e_{2358}-e_{1458}-e_{1368}.
\edm
We define a skew symmetric endomorphism $P(X,Y,.)$ on $TM$ via 
\bdm
g(P(X,Y,Z),V)\ =\ \Phi(X,Y,Z,V).
\edm
We extend the metric $g$ to $3$-forms on $TM$ in the usual way, 
i.\,e.~ $g(W_1\wedge W_2\wedge W_3,V_1\wedge V_2\wedge V_3)=\det(g(W_i,V_i))$ for 
$V_i,W_j\in TM$. For $3$-forms $\xi=\sum\limits_{i<j<k}\xi_{ijk}e_{ijk}$ and 
$\eta=\sum\limits_{i<j<k}\eta_{ijk}e_{ijk}$ let $\eta(\xi)$ be defined as 
\bdm
\eta(\xi):=\sum_{i<j<k}\xi_{ijk}\eta(e_i,e_j,e_k)=\sum_{i<j<k}\xi_{ijk}\eta_{ijk}=g(\eta,\xi).
\edm
We define $p(X)$ via 
\bdm
g(X,P(\xi))\ =\ g(p(X),\xi)
\edm 
for $X\in TM$ and a $3$-form $\xi$ on $M$ ($P(\xi)$ is well defined,
since $P$ is totally skew symmetric). $\Spin(7)$ manifolds were classified by 
Fern\'andez in \cite{F85}: they split in the two classes 
$\mathcal U_1$ and $\mathcal U_2$. S.~Ivanov proves in \cite{I04} 
that such a manifold always carries a characteristic connection.
%
\subsection{The cone correspondence for spinors}\label{subsec:cone-corr-spin}
%
Let the cone $(\bar M, \bar g)$ over $M$ with Levi-Civita connection 
$\bar\nabla^{\bar g}$ carry a $\bar G$ structure and assume that there is a 
connection $\nabla$ on $M$ such that its lift $\bar\nabla$ to $\bar M$ with 
torsion $\bar T$ is the characteristic connection on $\bar M$ with respect to 
the given $\bar G$ structure.\\

Given a $G$ structure on $M$, we shall construct an induced $\bar G$ structure 
on $\bar M$ in the following sections. We will see that the characteristic 
connection $\nabla^c$ on $M$ (with torsion $T^c$) does \emph{not} lift to the 
characteristic connection $\bar\nabla$ on $\bar M$ (with torsion $\bar T$, 
introduced by a connection $\nabla$ on $M$ with torsion $T$). In particular 
the lift $\overline{T^c}$ of the characteristic torsion to $\bar M$ is not the 
characteristic torsion on $\bar M$. So the tensor $T^c-T$ is not zero and will 
play an important role in the following.
We want to study the Killing equation with torsion as 
discussed in \cite{ABK12}: For the family of connections ($s \in \R$)
\bdm 
\nabla^s_XY\ =\ \nabla^g_XY+2sT^c(X,Y),
\edm
a spinor $\psi$ is called a \emph{Killing spinor with torsion} if it satisfies
the equation
\bdm
\nabla^s_X\psi\ =\ \alpha X\psi
\edm
for some Killing number $\alpha\in\R-\{0\}$ and some value of $s$. 
This definition includes the choice that we do not  view a parallel spinor 
($\alpha=0$) as a special case of a Killing spinor. A
priori, solutions of this equation with $\alpha\in\C-\R$ are conceivable,
but we are not aware of any. In any event, the cone construction would not
work for such an $\alpha$.

The case  $s=\frac{1}{4}$ corresponds to the 
characteristic connection; however, there are many geometric situations
in which the Killing equation holds for values $s\neq 1/4$.
The connection $\bar \nabla^s$ on $\bar M$ is then given by 
$\bar \nabla^s=\bar\nabla^{\bar g}+2s\bar T$. We obtain the following 
correspondence between connections on $\bar M$ and connections on $M$:

\bigskip
\begin{center}
\begin{tabular}{|c|c|}
 \hline
Connections on $M$ & Connections on $\bar M$\\
\hline
$\nabla^s=\nabla^g+2sT^c$ & $\bar\nabla^{\bar g} + 2s\overline{T^c} = \bar\nabla^s - 2s(\bar T-\overline{T^c})$\\
\hline
$\nabla^g+2sT=\nabla^s+2s(T-T^c)$&$\bar\nabla^s=\bar\nabla^{\bar g}+2s\bar T$\\
\hline
\end{tabular}
\end{center}
\bigskip

A direct application of Lemma \ref{lem:ks} implies: 
%
%
%
\begin{lem}\label{lem:corrksps}
For $\alpha\in\R-\{0\}$, we have the following correspondence between
\bigskip
\begin{center}
\begin{tabular}{|c|c|}
 \hline
spinors on $M$ & spinors on $\bar M$\\
\hline
$\nabla^s_X\psi=\alpha X\psi$ & $\bar\nabla^s_X\psi-sX\lrcorner(\bar T-\overline{T^c})\psi=0$\\
\hline
$\nabla^s_X\psi+sX\lrcorner(T-T^c)\psi=\alpha X\psi$ & $\bar\nabla^s_X\psi=0$\\
\hline
\end{tabular}
\end{center}
\bigskip

For $\dim(M)$ odd, 
there is one spinor on $M$ with either $\alpha=\frac{1}{2}a$ or 
$\alpha=-\frac{1}{2}a$.\\
If $\dim(M)$ is even, there is a pair of spinors with Killing numbers 
$\alpha=\pm\frac{1}{2}a$ on $M$.


%

In particular for $s=\frac{1}{4}$ we obtain the following correspondence:

\bigskip
\begin{center}
\begin{tabular}{|c|c|}
 \hline
spinors on $M$ & spinors on $\bar M$\\
\hline
$\nabla^c_X\psi=\alpha X\psi$ & $\bar\nabla_X\psi=\frac{1}{4}X\lrcorner(\bar T-\overline{T^c})\psi$\\
\hline
$\nabla^c_X\psi=\alpha X\psi-\frac{1}{4}X\lrcorner (T-T^c)~\psi$ & $\bar\nabla_X\psi=0$\\
\hline
\end{tabular}
\end{center}
\bigskip
\end{lem}

In the following sections we look at the corresponding structures on $\bar M$, 
their classifications and the correspondences of spinors on $M$ and $\bar M$.
%
%
\section{Metric almost contact structures -- almost hermitian 
structures on the cone}\label{sec:acms}
%
\subsection{Preparations}
Let $(M,g,\phi,\eta)$ be an $n$-dimensional metric almost contact structure.
As in Section \ref{ch:constrcone} we construct the twisted cone $\bar{M}$ over 
$M$ and define an almost hermitian structure $J$ on $\bar{M}$ via
\bdm
J(ar\del_r):=\xi,~~~~~ J(\xi):=-ar\del_r~~~~~ \mbox{ and } ~~~~~J(X)=-\phi(X) \mbox{ for } X \perp \xi,\del_r.
\edm
The identity $\phi^2=-\Id+\eta\otimes\xi$  immediately implies $J^2=-\Id$.
\begin{dfn}\label{def:normaldef}
If $M$ admits a characteristic connection $\nabla^c$ with skew symmetric 
torsion $T^c$ satisfying $\nabla^c\phi=\nabla^c\eta=0$, we define a connection 
$\nabla$ with skew symmetric torsion $T$ 
\begin{equation*}
T\ :=\ T^c-2a\eta\wedge F \mbox{ and thus } 
\nabla_XY\ =\ \nabla^c_XY-a(\eta\wedge F)(X,Y,.).
\end{equation*}
\end{dfn}
In particular: If the almost metric contact structure is Sasakian and
the Killing number happens to satisfy $|\alpha|=1/2$ (like
in the Riemannian case), the cone is constructed with $a=1$, and thus 
$T^c=\eta\wedge d\eta=2a\eta\wedge F$ and  $\nabla=\nabla^g$, the
Levi-Civita connection. Thus, $\nabla$ and $T$ measure in some sense  the
difference to the Riemannian Sasakian case.

Although the role of $T$ is clearly exposed in Section
\ref{subsec:cone-corr-spin}, this is not sufficient to determine $T$
completely. Rather, the formula for $T$ has to be found by trying
a suitable Ansatz, the motivation for which comes precisely from
the Riemannian case just described. Since $T$ is unique, the definition
is justified a posteriori by yielding the desired correspondence.

%
%
\begin{thm}\label{th:acs}
 If $(M,g,\phi, \eta)$ is an almost contact metric structure,
 $(\bar{M},\bar{g},J)$ is an almost hermitian manifold.\\
If furthermore  $M$ admits a characteristic connection, consider the
connection $\nabla$ defined above. Then the appendant connection 
$\bar{\nabla}$ on $\bar{M}$ is almost complex, $\bar{\nabla}J=0$.
\end{thm}
\begin{NB}
This shows in particular that $\bar\nabla$ is the unique characteristic 
connection of $\bar M$ with respect to $J$. Furthermore, the theorem includes
the claim that the existence of a characteristic connection for the
almost contact metric structure on $(M,g,\phi, \eta)$ suffices to imply
that the induced almost hermitian structure on $\bar M$ does also
admit a   characteristic connection.
\end{NB}
We first prove
\begin{lem}\label{lem:normal}
 On $M$, Definition $\ref{def:normaldef}$ implies 
\be\label{eq:normal}
(\nabla_Y\phi)X=ag(Y,X)\xi-a\eta(X)Y,
\ee
and  we have
\begin{enumerate}
 \item[$a)$] $a\phi(X)=-\nabla_X\xi$\label{eq:lem:normal:1},
 \item[$b)$] $\xi$ is a Killing vector field, $g(\nabla_Y\xi,X)=-g(\nabla_X\xi,Y)$ and thus its integral curves are geodesics,\label{eq:lem:normal:2}
 \item[$c)$] $d\eta=2aF+\xi\lrcorner T$.
\end{enumerate}
\end{lem}
\begin{proof}[Proof of Lemma $\ref{lem:normal}$]
Using the definition $\nabla=\nabla^c-a\eta\wedge F$ with the equation 
$\nabla^c\phi=0$, we directly compute $(\nabla_Y\phi)X=ag(Y,X)\xi-a\eta(X)Y$.
Identity (\ref{eq:normal}) and  $\phi(\xi)=0$ imply for $X\in TM$
\bdm
aX-ag(X,\xi)\xi\ =\ -(\nabla_X\phi)\xi\ =\ 
\nabla_X(\phi(\xi))-(\nabla_X\phi)\xi=\phi(\nabla_X\xi) .
\edm
Since $\nabla_X\xi\perp\xi$, applying $\phi$ yields
\bdm
a\phi(X)\ =\ \phi(aX-ag(X,\xi)\xi)\ =\ -\nabla_X\xi.
\edm
Since  $g(X,\phi(Y))=-g(\phi(X),Y)$, we can conclude from  
equation (\ref{eq:normal}) the statement b)
of the lemma, which is also  a consequence of 
Theorem 8.2 in \cite{FI02}.
For $X,Y \in TM$, we obtain with statement a) 
\begin{align*}
d\eta(X,Y)&\ =\ X\eta(Y)-Y\eta(X)-\eta([X,Y])\ =\ 
Xg(Y,\xi)-Yg(X,\xi)-g([X,Y],\xi)\\
&\ =\ g(\nabla_XY,\xi)+g(Y,\nabla_X\xi)-g(\nabla_YX,\xi)-g(X,\nabla_Y\xi)
-g([X,Y],\xi)\\
&\ =\ T(X,Y,\xi)-g(Y,a\phi(X))+g(X,a\phi(Y))\ =\ T(X,Y,\xi)+2aF(X,Y)
\end{align*}
which finishes the proof.
\end{proof}
\begin{proof}[Proof of Theorem $\ref{th:acs}$]
One easily checks that $\bar g(JX,JY)=\bar g(X,Y)$ for $X,Y\in T\bar{M}$ and 
thus $J$ is an almost hermitian structure.\\
 We have to show $\bar{\nabla}J=0$, meaning
 $0=\bar\nabla_Y(J(X))-J(\bar\nabla_YX)$. To do so, we distinguish the
 following cases:\\
If $X\in TM$, $X\perp \xi$ and $Y\in TM$ we have
\begin{align*}
\bar\nabla_Y(J(X))-J(\bar\nabla_YX)
&\ =\ -\bar\nabla(\phi(X))-J(\nabla_YX-\frac{1}{r}\bar g(Y,X)\del_r)\\
&\ =\ -\nabla_Y(\phi(X))+\frac{1}{r}\bar g(Y,\phi(X))
\del_r-J(\nabla_YX)+\frac{1}{ar^2}\bar g(Y,X)\xi\\
&\ =\ -(\nabla_Y\phi)(X)-\phi(\nabla_YX)+a^2rg(Y,\phi(X))\del_r-J(\nabla_YX)+ag(Y,X)\xi.
\end{align*}
With  identity (\ref{eq:normal}) and since $\eta(X)=0$, $\phi(\xi)=0$ we get
\begin{align*}
 \bar\nabla_Y(J(X))-J(\bar\nabla_YX)&\ =\ -a\eta(X)Y-\phi(\nabla_YX)
+a^2rg(Y,\phi(X))\del_r-J(\nabla_YX)\\
&\ =\ -\phi(\nabla_YX+ag(Y,\phi(X))\xi)-J(ag(Y,\phi(X))\xi+\nabla_YX),
\end{align*}
which is equal to zero if $\nabla_YX+ag(Y,\phi(X))\xi$ is perpendicular to 
$\xi$ and $\del_r$. Obviously it is perpendicular to $\del_r$. We have 
$g(\nabla_YX+ag(Y,\phi(X))\xi,\xi)=0$ if
\bdm
0\ =\ g(\nabla_YX,\xi)+g(Y,a\phi(X))\ =\ -g(X,\nabla_Y\xi)+g(Y,a\phi(X))
\ =\ g(X,a\phi(Y))+g(Y,a\phi(X))\ =\ 0.
\edm
If $X\in TM$, $X\perp\xi$ and $Y=\del_r$ we have 
$\bar\nabla_Y(J(X))-J(\bar\nabla_YX)=\frac{1}{r}J(X)-J(\frac{1}{r}X)=0$.\\
If $X=\xi$, $Y=\del_r$ we get 
\bdm
\bar\nabla_Y(J(X))-J(\bar\nabla_YX)
\ =\ \bar\nabla_{\del_r}(-ar\del_r)-J(\frac{1}{r}\xi)
\ =\ -a\del_r-ar\bar\nabla_{\del_r}\del_r+a\del_r=0.
\edm
Given $X=\xi$ and $Y=\xi$ we have 
\bdm
\bar\nabla_Y(J(X))-J(\bar\nabla_YX)
\ =\ -\bar\nabla_\xi(ar\del_r)-J(\nabla_\xi\xi-\frac{1}{r}\bar 
g(\xi,\xi)\del_r)\ =\ -a\xi+a\xi=0.
\edm
If $X=\xi$, $Y\in TM$, $Y\perp\xi$ we have 
\bdm
\bar\nabla_Y(J(X))-J(\bar\nabla_YX)
\ = \ -\bar\nabla_Y(ar\del_r)-J(\nabla_Y\xi-\frac{1}{r}
\bar g(Y,\xi)\del_r)=-aY+J(a\phi(Y))\ =\ -aY+aY\ =\ 0.
\edm
Given $X=\del_r$, $Y\perp\xi$, $Y\in TM$ we get 
\bdm
\bar\nabla_Y(J(X))-J(\bar\nabla_YX)\ =\ 
\bar\nabla_Y(\frac{1}{ar}\xi)-J(\frac{1}{r}Y)\ =\ 
-\frac{1}{ar}a\phi(Y)-J(\frac{1}{r}Y)=0.
\edm
In the case $X=\del_r$ and $Y=\xi$ we have 
\bdm
\bar\nabla_Y(J(X))-J(\bar\nabla_YX)\ =\ 
\bar\del_\xi(\frac{1}{ar}\xi)-J(\frac{1}{r}\xi)\ = \ 
\frac{1}{ar}\nabla_\xi\xi-\frac{1}{ar^2}\bar g(\xi,\xi)\del_r
+a\del_r\ =\ -a\del_r+a\del_r=0.
\edm
The last case is given by $X=Y=\del_r$. Then we have 
$\bar\nabla_{\del_r}(\frac{1}{ar}\xi)=-\frac{1}{ar^2}\xi+\frac{1}{ar}\bar\nabla_{\del_r}\xi=0$.
\end{proof}
Let $(M,g)$ be a Riemannian manifold such that the above constructed 
manifold $(\bar M,\bar g)$ carries an almost hermitian structure $J$. We 
have $J(\del_r)\perp \del_r$. We consider the manifold 
$M=M\times\{1\}\subset \bar M$ and define for 
$X\in TM$: $\xi:=aJ(\del_r)$, $\eta(X):=g(X,\xi)$ and 
$\phi(X):=-J(X)+\bar g(J(X),\del_r)\del_r$.
We get an almost contact structure on $M$:
\begin{align*}
\phi^2(X)\ =\ &-J(-J(X)+\bar g(J(X),\del_r)\del_r)+\bar g(J(-J(X)
+\bar g(J(X),\del_r)\del_r),\del_r)\del_r\\
 =\ &-X+\bar g(X,J(\del_r))J(\del_r)=-X+g(X,\xi)\xi=-X+\eta(X)\xi
\end{align*}
and 
\begin{align*}
g(\phi(X),\phi(Y))&\ =\ 
\frac{1}{a^2}\bar g(-J(X)+\bar g(J(X),\del_r)\del_r,-J(Y)
+\bar g(J(Y),\del_r)\del_r)\\
&\ =\ \frac{1}{a^2}(\bar g(J(X),J(Y))-\bar g(X,J\del_r)\bar g(Y,J(\del_r)))
\ = \ g(X,Y)-\eta(X)\eta(Y).
\end{align*}
Conversely to Theorem \ref{th:acs}, one proves:
\begin{thm}
Consider the manifold $\bar M$ equipped with a connection $\bar \nabla$ with 
skew symmetric torsion $\bar T$ being the lift of a connection $\nabla$ with 
torsion $T$ on $M$. If the connection $\bar \nabla$ is almost complex on 
$\bar M$, we have $(\nabla_X\phi)(Y)=ag(X,Y)\xi-a\eta(Y)X$ and thus the 
characteristic connection $\nabla^c$ on $M=M\times\{1\}$ has 
torsion $T^c=T+2a\eta\wedge F$.
\end{thm}
%
%
%
%
%
From now on we assume that $M$ and $\bar M$ admit an almost contact structure and
an almost hermitian structure, respectively, both admitting characteristic 
connections $\nabla^c$ and $\bar \nabla$ as introduced above.
%
%
\subsection{The classification of metric almost contact structures and the
  corresponding classification of almost hermitian structures on the cone}
%
We look at the classification of the geometric structures  on $\bar M$ and
$M$. We first prove the following two lemmata. 
\begin{lem}\label{lem:nijen}
The Nijenhuis tensor $\bar N$ of the almost hermitian structure on $\bar M$
restricted to $TM$ and the Nijenhuis tensor $N$ of the almost contact
structure on $M$ are related via $a^2r^2N=\bar N$. Furthermore, the following
conditions are equivalent:
\begin{itemize}
 \item $\del_r\lrcorner \bar N=0$,
 \item $d\eta(X,\phi Y)+d\eta(\phi X,Y)=0$ on $TM$,
 \item $\xi\lrcorner N=0$.
\end{itemize}
In particular $N=0$ if and only if $\bar N=0$.
\end{lem}
\begin{NB}\label{NB.others}
 In \cite{HTY12}, T.\,Houri, H.\,Takeuchi, and Y.\,Yasui considered 
hermitian manifolds $\bar M$ with a vanishing Nijenhuis tensor $\bar N$. 
They showed that in this case $N=0$ and thus $M$ is a normal almost 
contact manifold, which also is an immediate consequence of 
Lemma \ref{lem:nijen}. In \cite{Conti&M12}, D.\, Conti and Th.\,Madsen
investigated `Sasaki with torsion' manifolds, meaning normal
($N=0$) almost contact metric manifolds with $\xi$ a Killing vector field,
and their cones\,/\,cylinder; they obtained independently
the same result as Houri et.\,al. 
\end{NB}
\begin{NB}
Since $N=0$ if and only if $\bar N=0$, the condition $\bar N=0$ is sometimes 
used for the definition of an almost contact metric manifold to be 
normal (see for example \cite{CG90}).
\end{NB}
\begin{proof}[Proof of Lemma $\ref{lem:nijen}$]
Since we have
\bdm
\bar g((\bar\nabla^{\bar g}_XJ)Y,Z)\ =\ 
\bar g((\bar\nabla_XJ)Y+\frac{1}{2}(\bar JT(X,Y)-\bar T(X,JY)),Z)
\ = \ -\frac{1}{2}(\bar T(X,Y,JZ)+\bar T(X,JY,Z)),
\edm
the Nijenhuis tensor of $\bar M$ is given by
\begin{align*}
\bar N(X,Y,Z)&\ =\ \bar g((\bar\nabla^{\bar g}_XJ)(JY),Z)
-\bar g((\bar\nabla^{\bar g}_YJ)(JX),Z)+\bar g((\bar\nabla^{\bar
  g}_{JX}J)(Y),Z)-\bar g((\bar\nabla^{\bar g}_{JY}J)(X),Z)\\
&\ = \ \bar T(X,Y,Z)-\bar T(JX,JY,Z)-\bar T(JX,Y,JZ)-\bar T(X,JY,JZ),
\end{align*}
whereas the  Nijenhuis tensor on $M$ is
\begin{align*}
N(X,Y,Z)\ =\ &g((\nabla^g_X\phi)(\phi(Y))-(\nabla^g_Y\phi)(\phi(X))
+(\nabla^g_{\phi(X)}\phi)(Y)-(\nabla^g_{\phi(Y)}\phi)(X),Z)\\
&+\eta(X)g(\nabla^g_Y\xi,Z)-\eta(Y)g(\nabla^g_X\xi,Z).
\end{align*}
Identity (\ref{eq:normal}) implies
\bdm
 g((\nabla^g_X\phi)(Y),Z)\, =\, 
ag(X,Y)\eta(Z)-ag(X,Z)\eta(Y)-\frac{1}{2}(T(X,\phi(Y),Z)+T(X,Y,\phi(Z)))
\edm
and hence we obtain for $N(X,Y,Z)=$
\begin{align*}
&ag(X,\phi(Y))\eta(Z)- \frac{1}{2}T(X,\phi^2(Y),Z))
-\frac{1}{2}T(X,\phi(Y),\phi(Z))\\
&-ag(Y,\phi(X))\eta(Z)+ \frac{1}{2}T(Y,\phi^2(X),Z))
+\frac{1}{2}T(Y,\phi(X),\phi(Z))\\
&ag(\phi(X),Y)\eta(Z)-ag(\phi(X),Z)\eta(Y)
- \frac{1}{2}T(\phi(X),\phi(Y),Z))-\frac{1}{2}T(\phi(X),Y,\phi(Z))\\
&-ag(\phi(Y),X)\eta(Z)+ag(\phi(Y),Z)\eta(X)
+ \frac{1}{2}T(\phi(Y),\phi(X),Z))+\frac{1}{2}T(\phi(Y),X,\phi(Z))\\
&+\eta(X)g(\nabla^c_Y\xi,Z)-\frac{1}{2}\eta(X)T^c(Y,\xi,Z)
-\eta(Y)g(\nabla^c_X\xi,Z)+\frac{1}{2}\eta(Y)T^c(X,\xi,Z),
\end{align*}
which is the same as
\begin{align*}
=&\ T(X,Y,Z)-\frac{1}{2}\eta(Y)T(X,\xi,Z)-\frac{1}{2}\eta(X)T(\xi,Y,Z)
-T(X,\phi(Y),\phi(Z))\\
&-T(\phi(X),Y,\phi(Z))-T(\phi(X),\phi(Y),Z) -a\eta(Y)g(\phi(X),Z)
+a\eta(X)g(\phi(Y),Z)\\
&-\frac{1}{2}\eta(X)T(Y,\xi,Z)-\eta(X)aF(Z,Y)+\frac{1}{2}\eta(Y)T(X,\xi,Z)
+\eta(Y)aF(Z,X).
\end{align*}
For $X\in TM$ we have $J(X)+\eta(X)ar\del_r=J(X-\eta(X)\xi)
=-\phi(X-\eta(X)\xi)=-\phi(X)$. Since $\del_r\lrcorner \bar T=0$ for 
$X,Y,Z\in TM$ we  get 
\be\label{eq:JinT}
\bar T(J(X),Y,Z)= -a^2r^2T(\phi(X),Y,Z)
\ee
and also $\bar T(J(X),J(Y),Z)= a^2r^2T(\phi(X),\phi(Y),Z)$ etc. With this 
result we have
\bdm
N(X,Y,Z)=\frac{1}{a^2r^2}(\bar{T}(X,Y,Z)-\bar T(JX,JY,Z)-\bar T(JX,Y,JZ)
-\bar T(X,JY,JZ))
\edm
and thus we get the desired result $\bar N(X,Z,Z)=a^2r^2N(X,Y,Z)$ for $X,Y,Z\in TM$.\\
By definition of the Nijenhuis tensor we have $\del_r\lrcorner \bar N=0$ if 
and only if for $X,Y\in TM$  
\bdm
0=\bar T(\xi,JX,Y)+\bar T(\xi,X,JY) \Longleftrightarrow 0= T(\xi,\phi X,Y)+ T(\xi,X,\phi Y).
\edm
The relations  $\xi\lrcorner T=d\eta-2aF$ and 
\bdm
F(\phi X,Y)+F(X,\phi Y)\ =\ g(\phi X,\phi Y)+g(X,\phi^2Y)\ =\ 0
\edm
imply that
 $\del_r\lrcorner \bar N=0$ holds if and only if $d\eta(\phi X,Y)+d\eta(X,\phi Y)=0$.
In \cite{FI02} the identity $N(X,Y,\xi)=d\eta(X,Y)-d\eta(\phi X,\phi Y)$ is 
proved and we get
\bdm
N(\phi X,Y,\xi)\ =\ d\eta(\phi X,Y)+d\eta(X,\phi Y)-\eta(X)d\eta(\xi,\phi Y).
\edm
The identity $\xi\lrcorner T=d\eta-2aF$ implies $\xi\lrcorner d\eta=0$ and 
thus $d\eta(\phi X,Y)+d\eta(X,\phi Y)=0$ if and only if $N(\phi X,Y,\xi)=0$. 
Since $N$ is skew symmetric we have $N(\xi,Y,\xi)=0$ and thus 
$N(\phi X,Y,\xi)=0$ is equivalent to $\xi\lrcorner N=0$.
\end{proof}
\begin{lem}\label{lem:deltaO}
 For $Z\in T\bar M$ let $Z_M$ be the projection of $Z$ onto $TM$. Then we have
 \bdm
  \delta \Omega(Z)=-(\delta F -a(n-1)\eta)(Z_M).
 \edm
\end{lem}
\begin{proof}
 For $X,Y,Z\in T\bar M$ we have 
\bdm
(\bar\nabla^{\bar g}_X\Omega)(Y,Z)
=(\bar\nabla_X\Omega)(Y,Z)-\Omega(-\frac{1}{2}\bar{T}(X,Y),Z)
-\Omega(Y,-\frac{1}{2}\bar{T}(X,Z))=\frac{1}{2}(\bar T(X,JY,Z)+\bar T(X,Y,JZ)).
\edm
For a local ONB $\{e_1,..,e_n=\xi\}$ of $TM$ we get the local 
ONB $\{\bar e_1=\frac{1}{ar}e_1,..,\bar e_n
=\frac{1}{ar}e_n,\bar e_{n+1}=\del_r\}$ of $T\bar M$. In this basis 
and for $Z\in T\bar M$ we compute
\bdm
\delta \Omega(Z)=-\sum_{i=1}^{n+1}(\bar \nabla^{\bar g}_{\bar e_i}\Omega)
(\bar e_i,Z)= -\frac{1}{2}\sum_{i=1}^{n-1}
\bar T(\frac{1}{ar}e_i,\frac{1}{ar}Je_i,Z)
- \frac{1}{2}\bar T(\frac{1}{ar}\xi,-\del_r,Z)- \frac{1}{2}\bar T(\del_r,J\del_r,Z).
\edm
Since $\del_r\lrcorner \bar T=0$, with equation (\ref{eq:JinT}) and the 
fact that $\phi(e_n)=0$, we have
\begin{align*}
\delta \Omega(Z)&=\frac{1}{2}\sum_{i=1}^{n-1}T(e_i,\phi e_i,Z_M)
=\frac{1}{2}\sum_{i=1}^{n-1}(T^c(e_i,\phi e_i,Z_M)
-2a(\eta\wedge F)(e_i,\phi e_i,Z_M))\\
&=\frac{1}{2}\sum_{i=1}^{n-1}(T^c(e_i,\phi e_i,Z_M)
-2a\eta(Z_M)F( e_i,\phi e_i))=\frac{1}{2}\sum_{i=1}^{n}T^c(e_i,\phi e_i,Z_M)
+a\eta(Z_M)(n-1)\\
&=-(\delta F-a(n-1)\eta)(Z_M),
\end{align*}
finishing the proof.
\end{proof}
We consider the Gray-Hervella classification \cite{GH80} of almost hermitian 
structures, given in Section \ref{sec:som}. 
Since we want to work with characteristic connections, we will only consider 
structures of class $\chi_1\oplus\chi_3\oplus\chi_4$. We first translate the 
conditions of this classification for the almost hermitian structure on $\bar
M$ to conditions of the almost contact structure on $M$. For the discussion of
the classification of almost contact structures and the correspondences 
to the classification of almost hermitian structures see Theorem \ref{thm:acsvsacs}. 
\begin{thm}\label{thm:ocsclass}
We have the following correspondence between Gray-Hervella classes of almost
hermitian structures on the cone $\bar M$ and defining relations of
almost contact metric structures on $M$:

\medskip
\begin{center}
\begin{tabular}{|c|c|c|}
\hline
 Class of $\bar M$ & defining relation on $\bar M$ & corresponding relation on $M$\\
\hline
 K\"ahler & $\bar\nabla^{\bar g}J=0$& $(\nabla^g_XF)(Y,Z)=a\eta(Y)g(X,Z)$\\
 &&$-a\eta(Z)g(X,Y)$\\
\hline 
$\chi_3$&  $\delta \Omega=\bar N=0$ &$N=0$, $\delta F=a(n-1)\eta$\\
\hline
& $(\bar\nabla^{\bar g}_X\Omega)(Y,Z)=\frac{-1}{n-1}[\bar g(X,Y)\delta \Omega(Z)$
&$(\nabla^g_XF)(Y,Z)=\frac{\delta F(\xi)}{n-1}(g(X,Z)\eta(Y)$\\
$\chi_4$ &$-\bar g(X,Z)\delta \Omega(Y)-\bar g(X,JY)\delta \Omega(JZ)$&$-g(X,Y)\eta(Z))$ \\
&$+\bar g(X,JZ)\delta \Omega(JY)]$& \\
\hline
$\chi_1\oplus \chi_3$ &$\delta \Omega=0$&$\delta F=a(n-1)\eta$\\
\hline
$\chi_3\oplus \chi_4$&$\bar N=0$&$N=0$
\\
\hline
\end{tabular}
\end{center}
\medskip

Furthermore, a structure on $\bar M$ is never nearly K\"ahler (of class
$\chi_1$) 
nor of mixed class $\chi_1\oplus\chi_4$.
\end{thm}

%
\begin{proof}
We have
\bdm
a\eta(Y)g(X,Z)-a\eta(Z)g(X,Y)=a\eta\wedge F(X,Y,\phi Z)+a\eta\wedge F(X,\phi Y,Z).
\edm

\underline{K\"ahler case}:
Since the characteristic connection on $\bar M$ is unique, we have the
following equivalences
\bdm
\bar\nabla^{\bar g}J=0\Leftrightarrow \bar\nabla^{\bar g}=\bar \nabla
\Leftrightarrow\bar T=0\Leftrightarrow T=0 \Leftrightarrow T^c=2a\eta\wedge F.
\edm
For a metric connection $\tilde \nabla$ with skew symmetric torsion 
$\tilde T$ on $M$ one calculates
\bdm
(\tilde \nabla_XF)(Y,Z)\ = \ 
(\nabla^g_XF)(Y,Z)-\frac{1}{2}\tilde T(X,\phi Y,Z)-\frac{1}{2}\tilde T(X,Y,\phi Z).
\edm
Thus, $T^c=2a\eta\wedge F$ implies 
$(\nabla^g_XF)(Y,Z)=a\eta\wedge F(X,Y,\phi Z)+a\eta\wedge F(X,\phi Y,Z)$ and 
conversely the condition $(\nabla^g_XF)(Y,Z)=a\eta\wedge F(X,Y,\phi Z)
+a\eta\wedge F(X,\phi Y,Z)$ yields
\bdm
(\tilde \nabla_XF)(Y,Z)=(a\eta\wedge F-\frac{1}{2}\tilde T)(X,\phi Y,Z)
+(a\eta\wedge F-\frac{1}{2}\tilde T)(X,Y,\phi Z).
\edm
The uniqueness of the characteristic connection $\nabla^c$ on $M$ thus 
implies $T^c=2a\eta\wedge F$.\\

\underline{Case $\chi_3$}: Consider an almost hermitian structure on $\bar M$ 
of class $\chi_3$ defined by $\delta \Omega=\bar N=0$. 
With Lemma \ref{lem:nijen} and \ref{lem:deltaO} we have $\bar N=\delta
\Omega=0$ if and only if $N=0$ and $\delta F-a(n-1)\eta=0$.\\

\underline{Case $\chi_4$}: The defining relation for the class $\chi_4$ of an 
almost hermitian manifold $\bar M$
\bdm
(\bar\nabla^{\bar g}_X\Omega)(Y,Z)=\frac{-1}{n-1}[\bar g(X,Y)\delta \Omega(Z)-\bar g(X,Z)\delta \Omega(Y)-\bar g(X,JY)\delta \Omega(JZ)+\bar g(X,JZ)\delta \Omega(JY)]
\edm
 translates with Lemma \ref{lem:deltaO} for $X,Y,Z\in T\bar M$ into 
\begin{align*}
&\frac{1}{2}\bar T(X,Y,JZ)+\frac{1}{2}\bar T(X,JY,Z)=\\
&\frac{1}{n-1}[\bar{g}(X,Y)(\delta F(Z_M)-a(n-1)\eta(Z_M))
-\bar{g}(X,Z)(\delta F(Y_M)-a(n-1)\eta(Y_M))\\
&-\bar{g}(X,JY)(\delta F((JZ)_M)-a(n-1)\eta((JZ)_M))
+\bar{g}(X,JZ)(\delta F((JY)_M)-a(n-1)\eta((JY)_M))].
\end{align*}
For $X\in T\bar M$ we have $\bar g(\del_r,JX)=-\bar g(J\del_r,X)
=-\bar g(J\del_r,X_M)=-a^2r^2g(\frac{1}{ar}\xi,X_M)=-ar\eta(X_M)$ and 
for $X\in TM$ we have $(JX)_M=-\phi X$.\\
In the case where $X=\del_r$ and $Y,Z\in TM$, the defining relation is equivalent to
\bdm
0=ar\eta(Y)(\delta F(-\phi Z)-a(n-1)\eta(-\phi Z))
-ar\eta(Z)(\delta F(-\phi Y)-a(n-1)\eta(-\phi Y)).
\edm
This is satisfied if and only if $0=(\eta(Y)\delta F(\phi Z)
-\eta(Z)\delta F(\phi Y))=\eta\wedge (\delta F\circ \phi)(Y,Z)$. Taking
$Y=\xi$ we receive the condition $F\circ \phi=0$, which obviously is sufficient too.\\
If $X=Y=\del_r$, $Z\in TM$ the defining relation leads to
\bdm
0=\delta F(Z)-a(n-1)\eta(Z)-ar\eta(Z)(\delta F(\frac{1}{ar}\xi)-\frac{n-1}{r}),
\edm
which is the same as $0=\delta F(Z)-\eta(Z)\delta F(\xi)=-\delta F(\phi^2 Z)$, 
already being satisfied if $\delta F\circ \phi=0$.\\
The case $Y=Z=\del_r$ leads to $0=0$.\\
Given $Y=\del_r$ and $X,Z\in TM$ we get
\bdm
\frac{1}{2ar}\bar T(X,\xi,Z)=\frac{1}{n-1}[ar\eta(X)\delta F(-\phi (Z))
-a^2r^2F(X,Z)(\delta F(\frac{1}{ar}\xi)-a(n-1)\frac{1}{ar})].
\edm
Since we already have the condition $\delta F\circ \phi=0$ this is equivalent to
\bdm
d\eta(X,Z)-2aF(X,Z)=(\xi\lrcorner T)(X,Z)=\frac{2}{n-1}F(X,Z)(\delta F(\xi)-a(n-1)).
\edm
This is the same as $d\eta=\frac{2}{n-1}\delta F(\xi)F$.
At last we look at $X,Y,Z\in TM$. Again we already have $\delta F\circ \phi =0$
\bea[*]
\lefteqn{-\frac{1}{2}T(X,Y,\phi Z)-\frac{1}{2} T(X,\phi Y,Z)}\\
&=&\frac{1}{n-1}[g(X,Y)(\delta F(Z)-a(n-1)\eta(Z))-g(X,Z)(\delta F(Y)-a(n-1)\eta(Y))]\\
&=&g(X,Y)(\frac{\delta F}{n-1}-a\eta)(Z)-g(X,Z)(\frac{\delta F}{n-1}-a\eta)(Y).
\eea[*]
Furthermore we have
\bea[*]
-\frac{1}{2}(T(X,Y,\phi Z)+T(X,\phi Y,Z))&=&
-\frac{1}{2}\,T^c(X,Y,\phi Z)-\frac{1}{2}T^c(X,\phi Y,Z)
+a\eta(X)F(Y,\phi Z)+a\eta(Y)F(\phi Z,X)\\
& &+\,a\eta(X)F(\phi Y,Z)+a\eta(Z)F(X,\phi Y)\\
&=& -\,(\nabla^g_XF)(Y,Z)+a\eta(Y)g(\phi Z,\phi X)+a\eta(Z)g(X,\phi^2 Y)\\
&=& -\,(\nabla^g_XF)(Y,Z)+a\eta(Y)g(Z,X)-a\eta(Z)g(X,Y).
\eea[*]
Thus we get the equation
\bdm
(\nabla^g_XF)(Y,Z)=g(X,Z)\frac{\delta F}{n-1}(Y)-g(X,Y)\frac{\delta F}{n-1}(Z).
\edm
Since $\delta F\circ\phi=0$ we have $\delta F=\delta F(\xi)\eta$ and obtain
\bdm
(\nabla^g_XF)(Y,Z)
\eta(Y)-g(X,Y)(\frac{\delta F(\xi)}{n-1}+2a)\eta(Z)\\
=\frac{\delta F(\xi)}{n-1}(g(X,Z)\eta(Y)-g(X,Y)\eta(Z)).
\edm
We summarize this result: An almost hermitian structure on $\bar M$, given 
by an almost contact structure on $M$ is of class $\chi_4$ if and only if
\bdm
(\nabla^g_XF)(Y,Z)=\frac{\delta F(\xi)}{n-1}(g(X,Z)\eta(Y)-g(X,Y)\eta(Z)), ~~
\delta F\circ \phi=0 ~\mbox{ and }~ d\eta=2\frac{\delta F(\xi)}{n-1}F.
\edm
The first condition implies the others:
For some local orthonormal basis $e_1,\ldots ,e_n=\xi$ of $TM$ we have 
\begin{align*}
 \delta F(X)&=-\sum_{i=1}^n (\nabla^g_{e_i}F)(e_i,X)
=-\sum_{i=1}^n \frac{\delta F(\xi)}{n-1}(g(e_i,X)\eta(e_i)-\eta(X))\\
&=-\frac{\delta F(\xi)}{n-1}(-n\eta(X)+\eta(X))=\delta F(\xi)\eta(X)
\end{align*}
and thus the condition $(\nabla^g_XF)(Y,Z)=\frac{\delta F(\xi)}{n-1}(g(X,Z)
\eta(Y)-g(X,Y)\eta(Z))$ implies $\delta F\circ\phi=0$. Since $\xi$ is 
a Killing vector field and thus 
$(\nabla^g_XF)(\xi,\phi Y)=-F(\nabla^g_X\xi,\phi Y)=g(\nabla^g_X\xi,Y)$ is 
skew symmetric in $X$ and $Y$ we have
\begin{align*}
 d\eta(X,Y)=(\nabla^g_X\eta)(Y)-(\nabla^g_Y\eta)(X)
=(\nabla^g_XF)(\xi,\phi Y)-(\nabla^g_YF)(\xi,\phi X)=2(\nabla^g_XF)(\xi,\phi Y)
\end{align*}
and with condition $(\nabla^g_XF)(Y,Z)=\frac{\delta F(\xi)}{n-1}(g(X,Z)
\eta(Y)-g(X,Y)\eta(Z))$ we already get $d\eta=2\frac{\delta F(\xi)}{n-1}F$.\\

\underline{Case $\chi_1\oplus \chi_3$}: The condition for a structure of class
$\chi_1\oplus \chi_3$ can be obtained directly from Lemma \ref{lem:deltaO}.\\

\underline{Case $\chi_3\oplus \chi_4$}: An almost hermitian structure on 
$\bar M$ is of class $\chi_3\oplus \chi_4$ if and only if $\bar N=0$. Due 
to Lemma \ref{lem:nijen}, this is equivalent to $N=0$.\\

\Kommentar{
\underline{Case $\chi_1$}: $\bar M$ is nearly K\"ahler if and only if 
$(\bar\nabla^{\bar g}_XJ)X=0$ for all $X\in T\bar M$. This is equivalent to 
\bdm
\bar T(X,J(X),Y)=0 \mbox{ for all } X,Y\in T\bar M.
\edm
For $X=V+b\del_r$ and 
$V,W\in TM$ we get
\bdm
\bar T(V,J(V),W)+\frac{b}{ar}\bar T(V,\xi,W)=0 \mbox{ for any } V,W\in TM, b\in \R.
\edm
Hence we have $\xi\lrcorner \bar T=\xi\lrcorner T=0$ and leading to 
$d\eta=2aF+\xi\lrcorner T=2aF$ and thus $dF=0$.\\ 
%
%
Theorem 8.4 of \cite{FI02} tells us that this implies $N=0$ and the structure 
on $\bar M$ is also of class $\chi_3\oplus\chi_4$, thus is integrable.\\}

\underline{Case $\chi_1\oplus \chi_4$}: The condition for an almost hermitian 
structure to be of class $\chi_1\oplus \chi_4$ is the same as for the class 
$\chi_4$, setting $X=Y$:
\begin{align*}
\frac{1}{2}\bar T(X,JX,Y)= &\frac{1}{n-1}[\bar{g}(X,X)(\delta F(Y_M)-a(n-1)\eta(Y_M))
-\bar{g}(X,Y)(\delta F(X_M)-a(n-1)\eta(X_M))\\
&+\bar{g}(X,JY)(\delta F((JX)_M)-a(n-1)\eta((JX)_M))].
\end{align*}
The equation is still linear in $Y$ but not in $X$. We set $X=V+b\del_r$ 
for $b\in \R$ and $V\in TM$:
\begin{align*}
\frac{1}{2}\bar T(V,JV,Y)+\frac{b}{2ar}\bar T(V,\xi,Y)= 
&\frac{1}{n-1}[(b^2+a^2r^2g(V,V))(\delta F(Y_M)-a(n-1)\eta(Y_M))\\
&-(b\bar{g}(\del_r,Y)+a^2r^2g(V,Y_M))(\delta F(V)-a(n-1)\eta(V))\\
&+(\bar{g}(V,JY)-bar\eta(Y_M))(-\delta F(\phi V)+\frac{b}{ar}\delta F(\xi)
-\frac{b(n-1)}{r})].
\end{align*}
This is satisfied for any $b$ if and only if
\begin{align*}
\frac{1}{2}\bar T(V,JV,Y)= &\frac{1}{n-1}[a^2r^2g(V,V)(\delta F(Y_M)-a(n-1)\eta(Y_M))\\
&-a^2r^2g(V,Y_M)(\delta F(V)-a(n-1)\eta(V))+\bar{g}(V,JY)(-\delta F(\phi V))]
\end{align*}
and
\be\label{eq:W1W42}
\begin{split}
\frac{1}{2ar}\bar T(V,\xi,Y)= &\frac{1}{n-1}[-\bar{g}(\del_r,Y)(\delta F(V)
-a(n-1)\eta(V))\\
&+\bar{g}(V,JY)(\frac{\delta F(\xi)}{ar}-\frac{(n-1)}{r})+ar\eta(Y_M)\delta F(\phi V)]
\end{split}
\ee
and
\bdm
0=\delta F(Y_M)-\eta(Y_M)\delta F(\xi)
=\delta F(Y_M-\eta(Y_M)\xi)=-\delta F(\phi^2(Y_M)),
\edm
where the last equation is satisfied  if and only if $\delta F\circ\phi=0$. 

%
For $Y\in TM$ with the condition $\delta F\circ\phi=0$ 
%
equation (\ref{eq:W1W42}) leads to
\bdm
\frac{1}{2}T(\xi,V,Y)=F(V,Y)(\frac{\delta F(\xi)}{n-1}-a).
\edm
Since  $\xi\lrcorner T=d\eta-2aF$ we have $d\eta=2\frac{\delta F(\xi)}{n-1}F$ 
and thus $dF=0$.\\
With Theorem 8.4 in \cite{FI02} this implies $N=0$ and the structure is
already of class $\chi_4$. Thus a structure is never of class $\chi_1$ or of 
mixed class $\chi_1\oplus\chi_4$.
\end{proof}

We now compare the result of Theorem \ref{thm:ocsclass} with the $12$ 
classes of almost contact structures given in Section \ref{sec:som}. As in 
the whole article we just consider manifolds admitting a characteristic 
connections (recall that Theorem \ref{thm:acscharconn} formulates the
criterion for its existence).
\begin{thm}\label{thm:acsvsacs}
If the almost hermitian structure on $\bar M$ is
\begin{itemize}
\item of class $\chi_3$, then the almost contact structure on $M$ is 
of class $\mathcal{C}_3\oplus..\oplus\mathcal{C}_8$ but not of class 
$\mathcal{C}_3\oplus\mathcal{C}_4\oplus\mathcal{C}_5\oplus\mathcal{C}_7
\oplus\mathcal{C}_8$ or of class $\mathcal{C}_6$.
\item of class $\chi_1\oplus\chi_3$, then the almost contact 
structure on $M$ is not of class  $\mathcal{C}_1\oplus..\oplus\mathcal{C}_5
\oplus\mathcal{C}_7\oplus..\oplus\mathcal{C}_{12}$ nor of class $\mathcal{C}_6$.
\end{itemize}
The almost hermitian structure on $\bar M$ is
\begin{itemize}
\item K\"ahler if and only if the almost contact structure on $M$ is 
$\alpha$-Sasaki (of class $\mathcal{C}_6$) and $\delta F(\xi)=a(n-1)$.
\item of class $\chi_4$ if and only if the almost contact structure 
on $M$ is an $\alpha$-Sasaki structure.
\item of class $\chi_3\oplus\chi_4$ if and only if the almost contact 
structure on $M$ is of class $\mathcal{C}_3\oplus..\oplus\mathcal{C}_8$ and 
there exists a characteristic connection.
\end{itemize}
Furthermore the structure on $M$ is Sasaki if and only if the almost 
hermitian structure on $\bar M$ is of class $\chi_4$ with 
$\delta\Omega(\xi)=(a-1)(n-1)$.
\end{thm}
\begin{proof}
If the structure on $\bar M$ is of class $\chi_3$, we have $N=0$ and thus 
the structure on $M$ is of class $\mathcal{C}_3\oplus..\oplus\mathcal{C}_8$. 
Furthermore, $\delta F(\xi)=a(n-1)$ holds, but on 
$\mathcal{C}_3\oplus\mathcal{C}_4\oplus\mathcal{C}_5\oplus\mathcal{C}_7
\oplus\mathcal{C}_8$ we have $\delta F(\xi)=0$ and a structure on $M$ of 
class $\mathcal{C}_6$ implies a structure on $\bar M$ of class $\chi_4$.\\
A structure on $\bar M$ of class $\chi_1\oplus\chi_3$ implies  on $M$ 
the relation $\delta F(\xi)\neq0$, but on 
$\mathcal{C}_1\oplus..\oplus\mathcal{C}_5\oplus\mathcal{C}_7\oplus..\oplus
\mathcal{C}_{12}$ we have $\delta F(\xi)=0$ and again a structure on $M$ of 
class $\mathcal{C}_6$ implies a structure on $\bar M$ of class $\chi_4$.\\
With Theorem \ref{thm:ocsclass}, a structure on $\bar M$ is K\"ahlerian if and 
only if $(\nabla^g_XF)(Y,Z)=a\eta(Y)g(X,Z)-a\eta(Z)g(X,Y)$
holds on $M$,
which is equivalent for the almost contact structure to be of class 
$\mathcal{C}_6$ with $\delta F(\xi)=a(n-1)$.\\
The condition of Theorem \ref{thm:ocsclass} for a  structure of class $\chi_4$
on $M$ is equivalent to the definition of an almost contact structure on 
$\bar M$ to be of class $\mathcal{C}_6$.\\
In 
$\mathcal{C}_3\oplus..\oplus\mathcal{C}_8$ we have $N=0$, which together 
with the existence of a characteristic connection is equivalent to the 
property that the structure on $\bar M$ is of class $\chi_3\oplus\chi_4$.\\
A structure on $M$ is Sasaki if and only if it is of class $\mathcal{C}_6$ 
and $\delta F(\xi)=n-1$. Due to Theorem \ref{thm:ocsclass} this is equivalent 
to the condition for the structure on $\bar M$ to be of class $\chi_4$ with 
$\delta\Omega(\xi)=(a-1)(n-1)$.
\end{proof}
\begin{NB}\label{rem:acsriemann}
If we construct $\bar M$ with $a=1$, we obtain a K\"ahlerian structure, and
 $(\nabla^g_XF)(Y,Z)=\eta(Y)g(X,Z)-\eta(Z)g(X,Y)$ defines a Sasakian structure
 on $M$. This is the classical case treated by B\"ar in \cite{B93}.
\end{NB}
%
%
\subsection{Corresponding spinors on metric almost 
contact structures and their cones}
We shall now work out in detail the abstract spinor correspondence
stated in Lemma \ref{lem:corrksps} for the case that $M$ carries a metric
almost contact structure. The following result serves as a preparation.
\begin{lem}\label{lem:liftofetawedgeF}
Given a metric almost contact structure with characteristic connection on $M$,
the lift of $\eta\wedge F$ to its cone $\bar M$ is given by
\bdm
\frac{1}{a^3r^3}(\del_r\lrcorner\Omega)\wedge\Omega.
\edm
\end{lem}
\begin{proof}
Since $\del_r\lrcorner[\frac{1}{a^3r^3}(\del_r\lrcorner\Omega)\wedge\Omega]=0$ 
we just need to show the equality on $TM$. For $X,Y\in TM$  we have
\bdm
F(X,Y)=g(X,\phi Y)=-\frac{1}{a^2r^2}\bar g(X,JY+\eta(Y)ar\del_r)=-\frac{1}{a^2r^2}\Omega(X,Y)
\edm
and
\bdm
\eta(X)=g(X,arJ\del_r)=\frac{1}{ar}\Omega(X,\del_r)
\edm
which proves $F=-\frac{1}{a^2r^2}\Omega$ and 
$\eta=-\frac{1}{ar}\del_r\lrcorner\Omega$ on $TM$.
\end{proof}
We recall the definition of the connections
\bdm 
\nabla^s_XY\, =\, \nabla^g_XY+2sT^c(X,Y) \mbox{ and } 
\bar\nabla^s_XY\, =\, \bar\nabla^{\bar g}_XY+2s\bar T(X,Y)
\edm
 for $s \in \R$ from Section \ref{ch:constrcone}.
Theorem \ref{th:acs} yields $T^c=T+2a\eta\wedge F$ and since $\bar T=a^2r^2T$
and $\overline{T^c}=a^2r^2T^c$, we get $\overline{T^c}-\bar T$ as the lift of 
$2a^3r^2\eta\wedge F$ to $\bar M$. With Lemma \ref{lem:liftofetawedgeF} we
obtain $\overline{T^c}-\bar T=\frac{2}{r}(\del_r\lrcorner\Omega)\wedge\Omega$.
\begin{thm}\label{th:ksacs}
Assume that the  almost contact metric  manifold $(M,g,\phi,\eta)$ admits
a characteristic connection and is spin. Then there is for 
$\alpha=\frac{1}{2}a$ or $\alpha=-\frac{1}{2}a$:
\begin{enumerate}
 \item A one to one correspondence between Killing spinors with torsion \label{thm:item1th}
\bdm
\nabla^s_X\psi=\alpha X\psi
\edm
on $M$  and parallel 
spinors of the connection $\bar\nabla^s+\frac{4s}{r}(\del_r\lrcorner\Omega)
\wedge\Omega$ on  $\bar M$ with cone constant $a$
\bdm
\bar\nabla^s_X\psi+\frac{2s}{r}(X\lrcorner(\del_r\lrcorner\Omega)\wedge\Omega)\psi=0,
\edm
\item A one to one correspondence between $\bar\nabla^s$-parallel spinors on $\bar M$ 
with cone constant $a$  and spinors on $M$ satisfying
\bdm
\nabla^s_X\psi-2asX\lrcorner(\eta\wedge F)\psi=\alpha X\psi.
\edm
\end{enumerate}
In particular, for $s=\frac{1}{4}$ we get the correspondence
\begin{center}
\begin{tabular}{|c|c|}
 \hline
spinors on $M$ & spinors on $\bar M$\\
\hline
$\nabla^c_X\psi=\alpha X\psi$ & $\bar\nabla_X\psi
=-\frac{1}{2r}X\lrcorner((\del_r\lrcorner\Omega)\wedge\Omega)\psi$\\
\hline
$\nabla^c_X\psi=\alpha X\psi+\frac{a}{2}X\lrcorner (\eta\wedge F)\psi$ 
& $\bar\nabla_X\psi=0$\\
\hline
\end{tabular}
\end{center}
\end{thm}
\medskip

%
\begin{NB}\label{rem:chartorausschr}
 Since $\bar\nabla=\bar\nabla^{\bar g}+\frac{1}{2}\bar T$ is the
 characteristic connection of the almost hermitian structure on $\bar M$, we can write
\bdm
\bar T\ =\ \bar N+d\Omega^J,
\edm
where $d\Omega^J =d\Omega\circ J$. Thus one can rewrite all equations 
above. For example the correspondence (\ref{thm:item1th}) of Theorem 
\ref{th:ksacs} is given with spinors on $\bar M$ satisfying
\bdm
\bar\nabla^{\bar g}_X\psi
+sX\lrcorner[\bar
N+d\Omega^J+\frac{2}{r}(\del_r\lrcorner\Omega)\wedge\Omega]\psi\, =\, 0.
\edm
Equivalently, one can use the description of $T^c$ on $M$ given by 
$T^c=\eta\wedge d\eta+dF^\phi+N-\eta\wedge(\xi\lrcorner N)$ (\cite{FI02}) to
rewrite the second correspondence.
Note that this also implies that $\bar T=\bar N+d\Omega^J$ is the lift of 
\bdm
a^2r^2T=a^2r^2(T^c-2a\eta\wedge F)=a^2r^2(\eta\wedge
(d\eta-2aF)+dF^\phi+N-\eta\wedge(\xi\lrcorner N))
\edm
to $\bar M$, in particular we have $\del_r\lrcorner(\bar N+d\Omega^J)=0$. 
\end{NB}
\subsection{Examples}\label{subsection.exas}
%
In this Section, we shall discuss several examples of metric almost contact
structures and the special spinor fields that exist on them and on their
cones. In particular, we shall describe sereval situations where the cone carries
a parallel spinor field for the characteristic connection $\bar\nabla$
of its almost hermitian structure.
\begin{exa}
For a metric almost contact manifold $(M,g,\phi,\eta)$, the deformation
\bdm
g_t:=tg+(t^2-t)\eta\otimes\eta,~~ \xi_t:=\frac{1}{t}\xi,~~ \eta_t:=t\eta,
\quad t>0
\edm
is often used for different purposes and constructions. 
It was  introduced by Tanno \cite{T68}, which explains why it is either
called \emph{Tanno deformation} or \emph{D-homothetic deformation}.
It has the property that if the original manifold is K-contact or Sasaki, then
the deformed manifold $(M,g_t,\xi_t,\eta_t,\phi)$ has again this
property. 

In  \cite[Cor.2.18]{BB12} it was proved that any Sasakian $\eta$-Einstein manifold
(with certain weak relations between the curvature parameters) 
carries  Killing spinors with torsion,
while Einstein-Sasaki manifolds can never admit Killing spinors with 
non trivial torsion \cite{ABK12}. Since any $\eta$-Einstein manifold 
can be Tanno deformed into
an Einstein manifold (\cite{T67}, \cite{T68}), it is thus sufficient to restrict
our attention to Tanno deformations of Einstein-Sasaki manifolds.
It is well-known that these carry Riemannian Killing spinors  \cite{FK89}.

In \cite{BB12}, the  Killing spinors with torsion on the
 Tanno deformation of an
Einstein-Sasaki manifold $(M,g,\phi,\eta)$ of dimension $n=2k+1\geq 5$
are constructed as follows.
Consider the  one dimensional subbundles of the spinor bundle $\Sigma_t$ 
of $(M,g_t)$ defined by
\bdm
 L_1(\Sigma_t):=\{\psi\in\Sigma_t~|~\phi(X)\psi=-iX\psi~ 
\forall X\perp\xi\}, ~~~ L_2(\Sigma_t):=\{\psi\in\Sigma_t~|~\phi(X)\psi
=iX\psi~ \forall X\perp\xi\}.
\edm
Define $\epsilon=\pm 1$ to be the number satisfying 
$e_1\phi(e_1)...e_k\phi(e_k)\xi\psi=\epsilon i^{k+1}\psi$ for a local
orthonormal frame $e_1,\phi(e_1),..,e_k,\phi(e_k),\xi$ on $M$. 
Theorem $2.22$ from \cite{BB12} then states that the spinors $\psi_1\in L_1(\Sigma_t)$ 
and $\psi_2\in L_2(\Sigma_t)$ are Killing spinors with torsion for 
$s_t=\frac{k+1}{4(k-1)}(\frac{1}{t}-1)$ with Killing numbers 
\bdm\tag{$*$}
\beta_{1,t}=\frac{\epsilon}{2}\frac{2kt-(k+1)}{t(k-1)}
= \frac{\epsilon}{2}(1-4s_t) ~~~~~\mbox{ and }~~~~~ 
\beta_{2,t}= (-1)^{k+1}\beta_{1,t} 
\edm
respectively. For $t=1$, there is no deformation, and indeed
the parameter $s_t$ is then zero and the two spinors are just
classical Riemannian Killing spinors.
Since $(M,g_t,\xi_t,\eta_t,\phi)$ with fundamental $2$-form $F_t$ is Sasakian, 
the characteristic
torsion of $\nabla^c$ is given by $T^c = \eta_t\wedge d\eta_t=2\eta_t\wedge
F_t$. Thus, the Killing equation
\bdm
\nabla^{g_t}_X\psi_i + s_t (X \lrcorner T^c ) \psi_i\ =\ \beta_{i,t} X \psi_i,\quad
i=1,2
\edm
can equivalently be reformulated as 
\bdm
\nabla^{g_t}_X\psi_i + \frac{1}{4} (X \lrcorner T^c ) \psi_i -
(1-4 s_t)\frac{1}{4} (X \lrcorner T^c ) \psi_i\ =\ \beta_{i,t}X \psi_i.
\edm
If $1-4s_t=0$,  both Killing numbers $\beta_{i,t}$ vanish by equation $(*)$ 
and the Killing equation
is reduced to $\nabla^c\psi_i=0$ -- the spinor fields $\psi_i$ are
$\nabla^c$-parallel and, as observed before, the cone construction is
not possible. The condition   $1 - 4s_t>0$ is equivalent to $t>\frac{k+1}{2k}$
and we observe that in this case, the last equation is exactly of the
form treated in Theorem \ref{th:ksacs}, case (2) for $s=1/4$ and
$a=2|\beta_{i,t}| = 1 - 4s_t>0$. Recall that we know from Theorem
$\ref{thm:acsvsacs}$ that the cone $(\bar M, \bar g_t)$ of the Tanno
deformation is a locally conformally K\"ahler manifold (class $\chi_4$).
Hence, we can conclude from Theorem \ref{th:ksacs}, case (2):
\begin{thm}\label{thm.Tanno-deformation}
Let $(M,g,\phi,\eta)$ be an Einstein Sasaki manifold
of dimension $2k+1\geq 5$. Consider its Tanno deformation 
$(M,g_t,\xi_t,\eta_t,\phi)$  for  $t> \frac{k+1}{2k}$ and the cone
$(\bar M, \bar g_t, J_t)$, constructed with 
cone constant $a= 1 - 4s_t>0$, and endowed with the conformally
K\"ahler structure   described before.
Then the two Killing spinors with torsion on $(M,g_t,\xi_t,\eta_t,\phi)$
induce each a spinor on the cone $(\bar M, \bar g_t, J_t)$ that is
parallel with respect to its characteristic connection $\bar\nabla$.
\end{thm}
Although Killing spinors with torsion do exist on $(M,g_t,\xi_t,\eta_t,\phi)$  
for  $0<t<\frac{k+1}{2k}$, Theorem \ref{th:ksacs}, case (2) cannot be applied
because the signs do not match. Of course, case (1) does still hold and
therefore we obtain a spinor field satisfying a more complicated equation on
$\bar M$. For $t=1$ (meaning $s_t=0$), Theorem \ref{thm.Tanno-deformation}
is the classical cone correspondence between Riemannian Killing spinors on 
Einstein-Sasaki manifolds and
Riemannian parallel spinors on their cone \cite{B93}.
\end{exa}
%
%
%
%
\begin{exa}
We shall now prove the existence of parallel spinors on the cone for a manifold
that is not Sasaki and that cannot be deformed into a manifold carrying
Riemannian Killing spinors.
The Heisenberg group $H$ is defined to be the following Lie subgroup of $\Gl(4,\R)$:
\bdm
H:=\left\{ \begin{bmatrix} 1 &u&v&z\\0&1&0&x\\0&0&1&y\\0&0&0&1 \end{bmatrix}~:~ u,v,x,y,z\in\R  \right\}.
\edm
The vector fields $u_1=\del_u$, $u_2=\del_x+u\del_z$, $u_3=\del_v$, 
$u_4=\del_y+v\del_z$, and $u_5=\del_z$ form a basis of the left invariant 
vector fields. For $\rho>0$ we  consider the metric (\cite{KV85})
\bdm
g_\rho \ =\ \frac{1}{\rho}(du^2+dx^2+dv^2+dy^2)+(dz-udx-vdy)^2
\edm
and get an orthonormal frame $e_1=\sqrt{\rho}u_1$, $e_2=\sqrt{\rho}u_2$, 
$e_3=\sqrt{\rho}u_3$, $e_4=\sqrt{\rho}u_4$ and $e_5=u_5$.
On $H$, there exists a left-invariant spin structure such that 
$e_1e_2e_3e_4e_5\psi=i\psi$ for all spinor fields $\psi$, which 
is the one we choose.
We consider the almost contact structures given by $\xi:=e_5$ and
the fundamental $2$-forms
\bdm
F_1:=e_1\wedge e_2-e_3\wedge e_4 \mbox{ and } F_2:= -(e_1\wedge e_2+e_3\wedge e_4).
\edm
It is a lengthy, but routine calculation to determine the
class of these metric almost contact structures. Together with
Theorem \ref{thm:acsvsacs}, the final result is:
\begin{lem}
\begin{enumerate}
\item[]
\item $(H,g_\rho)$ is never an Einstein manifold $\forall \rho>0$ and its Tanno
deformation is again a metric in the same family of metrics.
\item
The structure $F_1$ is of class $\mathcal{C}_7$ and the structure 
$F_2$ is of class $\mathcal{C}_6$ (for $\rho=2$, $F_2$ is Sasakian).
\item The almost hermitian structure on $\bar M$ induced by $F_1$ 
is hermitian (mixed class $\chi_3\oplus\chi_4$)
and the almost hermitian structure on $\bar M$ induced by $F_2$ 
is locally conformally K\"ahler (class $\chi_4$). 
With respect to the orthonormal frame $X_i:=\frac{1}{ar}e_i$ for 
$i=1,\ldots , 5$ and $X_6:=\del_r$, they are given by
\bdm
\Omega_1=-X_1\wedge X_2+X_3\wedge X_4+X_5\wedge X_6 \mbox{ and } 
\Omega_2=X_1\wedge X_2+X_3\wedge X_4+X_5\wedge X_6.
\edm
\end{enumerate}
\end{lem}
In particular, $N_i=\bar N_i=0$ and $dF_i=0$ for $i=1,2$. 
Becker-Bender calculates in \cite{BB12} that the characteristic connection 
for both structures is given by 
$T^c=-\rho (e_1\wedge e_2 +e_3\wedge e_4)\wedge e_5$. 
One checks that $d\eta=-\rho (e_1\wedge e_2 +e_3\wedge e_4)$, hence
$d\eta=\rho F_2$, whereas $F_1$ is not proportional to $d\eta$.
She also 
proves that $\psi_1$ and $\psi_2$, defined via the equations
\bdm
\phi_2(X)\psi_1=-iX\psi_1 ~\forall X\perp \xi \mbox{ and } 
\phi_2(X)\psi_2=iX\psi_2 ~\forall X\perp \xi,
\edm
where $\phi_j$ is the $(1,1)$ tensor to the $2$-form $F_j$ for $j=1,2$, 
are Killing spinors with torsion for $s=-\frac{3}{4}$ with Killing number 
$\rho$ and $-\rho$ respectively:
\bdm
\nabla^{-\frac{3}{4}}_X\psi_1=\rho X\psi_1 \mbox{ and } 
\nabla^{-\frac{3}{4}}_X\psi_2=-\rho X\psi_2 .
\edm
If we set $\rho_1=\rho, \rho_2=-\rho$, we can rewrite these equations as
\bdm
\nabla^c_X\psi_i -(X\lrcorner T^c)\psi_i\ =\ \rho_i X\psi_i.
\edm
On the other hand, let us consider again the equation from
Theorem \ref{th:ksacs}, case (2), for $s=1/4$:
\bdm
\nabla^c_X\psi-\frac{a}{2}X\lrcorner(\eta\wedge F)\psi=\alpha X\psi.
\edm
Since $a$ has to be chosen as $a=2|\alpha|=2|\rho_i|=2\rho$, we conclude
that both Killing spinors $\psi_1,\psi_2$ with torsion on the Heisenberg group
satisfy this equation for the structure $F=F_2$.
Therefore, their lifts to the cone are parallel for the characteristic
connection of the conformally K\"ahler structure $\Omega_2$.
We see at once that the argument can be generalized as follows:
\begin{lem}
Let $(M,g,\phi,\eta)$ be an $\alpha$-Sasaki structure (class $C_6$)
satisfying $d\eta=\lambda F$ for some $\lambda>0$ and admitting
a Killing spinor with torsion with Killing number
$\alpha=\lambda$ or $\alpha=-\lambda$ for $s=-3/4$ . Then its cone is a
locally conformally K\"ahler manifold (class $\chi_4$), and the spinor
lifts to a parallel spinor on $\bar M$ with respect to its characteristic
connection.
\end{lem}
Let us have a closer look at the characteristic connections $\bar\nabla^i$, induced
by the connections $\nabla^i$ with torsions $T^i=T^c-2a\eta\wedge F_i$ on $M$, 
and the $s$-dependent connections $\bar\nabla^{s,i}:=\bar\nabla^g+2s\bar
T^i$ ($i=1,2$). Since $F_1\neq F_2$,
we see that the characteristic connections  $\bar\nabla^i$ (of the almost 
hermitian structures  $\Omega_i$) on $\bar M$ do not coincide,
despite the fact that the characteristic connections  (of the metric
almost structures $F_i$)  coincide on $M$, $i=1,2$. This illustrates 
neatly the subtle dependence of the construction on the underlying
geometric structure, not only its characteristic connection.

The equivalence of the 
characteristic connections for $F_1$ and $F_2$ on $M$ implies that the 
connections $\bar\nabla^{s,i}+\frac{4s}{r}(\del_r\lrcorner\Omega_i)\wedge\Omega_i$ 
are the same for $i=1,2,\ s=-3/4$. As discussed above, this
connection is in turn just the characteristic connection of the locally
conformally K\"ahler structure $\Omega_2$, hence we have the following
relation between the K\"ahler forms:
\bdm
d\Omega_2^{J_2}\ =\ - 3\,[d\Omega_i^{J_i}+ \frac{2}{r} (\del_r\lrcorner \Omega_i)\wedge \Omega_i]\ \qquad i=1,2.
\edm
In particular, we can apply Theorem \ref{th:ksacs},  case (1) for $i=1$ 
and can state that the differential equation for the two  $\bar\nabla^2$-parallel 
spinors on $\bar M$ can equally be written
\bdm
0\ =\ \bar\nabla^{-\frac{3}{4},1}_X\psi - \frac{3}{2r} 
X\lrcorner((\del_r\lrcorner\Omega_1)\wedge\Omega_1)\psi \ =\
\bar\nabla^{\bar g}_X\psi-\frac{3}{4}X\lrcorner[d\Omega_1
+\frac{2}{r}(\del_r\lrcorner\Omega_1)\wedge\Omega_1]\psi.
\edm
\end{exa}
\begin{exa}
Another example (see \cite{BB12}) is given by the homogeneous 
space $M:=\SO(3)\times\SL(2,\R)/\SO(2)$ with the embedding
\bdm
\SO(2)\ni A(t):=\begin{bmatrix} \cos t & -\sin t\\ \sin t &\cos t
\end{bmatrix}  \mapsto \left[A(t),A\left(\frac{t}{2}\right)^{-1}\right].
\edm
As an orthonormal basis of a reductive complement of $\so(2)$ in 
$\so(3)\times\sl(2,\R)$ we choose
\bdm
e_1\ :=\ D_1 \left(\begin{bmatrix} 0&0&0\\0&0&-1\\0&1&0
  \end{bmatrix},0\right),\quad
e_2:=D_1\left(\begin{bmatrix} 0&0&1\\0&0&0\\-1&0&0
  \end{bmatrix},0\right),\quad
e_5:=\left(c_1\begin{bmatrix} 0&-1&0\\1&0&0\\0&0&0 
\end{bmatrix},c_2\frac{1}{2}\begin{bmatrix} 0&-1\\1&0 \end{bmatrix}\right),
\edm
\bdm
e_3:=\frac{1}{2}D_2\left(0,\begin{bmatrix} 1&0\\0&-1 \end{bmatrix}\right),\quad
e_4:=\frac{1}{2}D_2\left(0,\begin{bmatrix} 0&1\\1&0 \end{bmatrix}\right),
\edm
such that $c_1+c_2\neq0$, $D_1^2=c_1(c_1+c_2)$, $D_2^2=-c_2(c_1+c_2)$ and 
the numbers $c_1$, $-c_2$ and $(c_1+c_2)$ have the same signature. We consider 
the almost contact structure $(M,\xi,F)$ defined via
\bdm
\xi:=e_5 \mbox{ and } F=e_1\wedge e_2+e_3\wedge e_4.
\edm
Then the characteristic connection $\nabla^c$ has torsion 
$T^c=-c_1e_1\wedge e_2\wedge e_5-c_2e_3\wedge e_4\wedge e_5$.
\begin{lem}
 The almost contact structure $(M,\xi,F)$ is normal. Furthermore, the 
almost hermitian structure on $\bar M$, constructed with
$a=\frac{-c_1-c_2}{4}$, induced by the almost contact structure $(M,\xi,F)$ 
is of class $\chi_3$ and thus the structure $(M,\xi,F)$ is of mixed class 
$\mathcal{C}_3\oplus..\oplus\mathcal{C}_8$.
\end{lem}
\begin{proof}
One uses Theorem \ref{thm:ocsclass} and proves that the almost contact 
structure $(M,\xi,F)$ is normal, satisfies $\delta F=(-c_1-c_2)\eta$ and 
never satisfies $(\nabla^g_XF)(Y,Z)=a\eta(Y)g(X,Z)-a\eta(Z)g(X,Y)$. Thus, 
the structure on $\bar M$ is never K\"ahler and for $a=\frac{-c_1-c_2}{4}$ 
it really is of class $\chi_3$.
\end{proof}
In this example we only have Killing spinors with torsion satisfying 
$\nabla^s_X\psi=\alpha X\psi$ for $\alpha=0$. But since the construction of 
$\bar M$ explicitly depends on $2\alpha=a\neq0$, we cannot lift these spinors 
to $\bar M$.
\end{exa}
%
%
%
\subsection{Metric almost contact $3$-structures}
%
%
%
Let $M$ be a manifold of dimension $n=4m-1$ with $3$ metric almost contact 
structures given by $\xi_i$, $\eta_i$ and $\phi_i$ for $i=1,2,3$. Looking at 
the cone $\bar M$, we define the three almost hermitian structures 
\begin{align*}
&J_1(ar\del_r):=\xi_1,  &J_1(\xi_1)=-ar\del_r,  \qquad\qquad 
&J_1(V)=-\phi_1(V) \mbox{ for } V\perp \xi_1,\del_r,\\
&J_2(ar\del_r):=\xi_2,  &J_2(\xi_2)=-ar\del_r, \qquad\qquad  
&J_2(V)=-\phi_2(V) \mbox{ for } V\perp \xi_2,\del_r,\\
&J_3(ar\del_r):=-\xi_3, &J_3(\xi_3)=ar\del_r,  \qquad \qquad
&J_3(V)=-\phi_3(V) \mbox{ for } V\perp \xi_3,\del_r.
\end{align*}
Conversely, let $\bar M$ be a $4m$ dimensional manifold with three almost 
hermitian structures $J_1$, $J_2$ and $J_3$. We can define three almost 
contact structures 
\begin{alignat*}{2}
\xi_1&\ :=\ +aJ_1(\del_r),\qquad   &\phi_1(X) &\ := \ 
-J_1(X)+\bar g(J_1(X),\del_r)\del_r.\\
\xi_2&\ :=\ +aJ_2(\del_r), \qquad   &\phi_2(X) &\ :=\ 
-J_2(X)+\bar g(J_2(X),\del_r)\del_r.\\
\xi_3&\ :=\ -aJ_3(\del_r),\qquad    &\phi_3(X) &\ :=\ 
+ J_3(X)-\bar g(J_3(X),\del_r)\del_r.
\end{alignat*}
on $M=M\times\{1\}\subset \bar M$. We can apply Theorem \ref{th:acs} to 
each of these structures and prove
\begin{thm}\label{thm:3aks}
 The three almost hermitian structures on $\bar M$ satisfy the relation 
$J_1J_2=-J_2J_1=J_3$ if and only if $\xi_1$, $\xi_2$ and $\xi_3$ are 
orthonormal and the almost contact structures on $M$ satisfy the following
\begin{align}\label{eq:hkc1}
\phi_3\phi_2\ =\ -\phi_1+\eta_2\otimes\xi_3,\quad
\phi_2\phi_3\ =\ +\phi_1+\eta_3\otimes\xi_2,\quad
\phi_1\phi_3\ =\ -\phi_2+\eta_3\otimes\xi_1,\\
\phi_3\phi_1\ =\ +\phi_2+\eta_1\otimes\xi_3,\quad
\phi_2\phi_1\ =\ -\phi_3+\eta_1\otimes\xi_2,\quad
\phi_1\phi_2\ =\ +\phi_3+\eta_2\otimes\xi_1,\label{eq:hkc2}
\end{align}
where $\eta_i$ is the dual to $\xi_i$ for $i=1,2,3$.
If and only if there are characteristic connections $\nabla^{c,i}$ on $M$ 
for each of the three almost hermitian structures $(\eta_i,\,\phi_i)$
such that the corresponding connections $\nabla^i$ constructed in 
Definition $\ref{def:normaldef}$ coincide $\nabla^1=\nabla^2=\nabla^3=:\nabla$, 
we have for the appendant connection $\bar\nabla$: 
$\bar\nabla J_2=\bar\nabla J_3=\bar\nabla J_1=0$. In this case additionally 
we get the commutator relations 
\bdm
[\xi_1,\xi_2]=2a\xi_3-T(\xi_1,\xi_2),\quad
[\xi_2,\xi_3]=2a\xi_1-T(\xi_2,\xi_3),\quad
[\xi_3,\xi_1]=2a\xi_2-T(\xi_3,\xi_1).
\edm
\end{thm}
\begin{proof}
 Given three almost hermitian structures satisfying the relation 
$J_1J_2=-J_2J_1=J_3$, we compute
\begin{align*}
 \phi_3(\phi_2(X))&=-J_3(J_2(X))+\bar g(J_3(J_2(X)),
\del_r)\del_r+\bar g(J_2(X),\del_r)J_3(\del_r)-\bar g(J_3(X), \del_r)\bar g(J_3(\del_r),\del_r)\del_r\\
&=-\phi_1(X)-\bar g(X,J_2\del_r)J_3(\del_r)=-\phi_1(X)-a^2g(X,J_2\del_r)J_3(\del_r)\\
&=-\phi_1(X)+g(X,\xi_2)\xi_3,
\end{align*}
and similarly for the other relations. Conversely, 
given three almost hermitian structures satisfying equations (\ref{eq:hkc1}) 
and (\ref{eq:hkc2}) we plug in $\xi_1$, $\xi_2$, and $\xi_3$ and, with 
$\phi_i(\xi_i)=0$ for $i=1,2,3$, we obtain immediately
\bdm
\phi_1(\xi_2)=\xi_3,\quad
\phi_1(\xi_3)=-\xi_2,\quad
\phi_2(\xi_1)=-\xi_3, \quad
\phi_2(\xi_3)=\xi_1, \quad
\phi_3(\xi_1)=\xi_2, \quad
\phi_3(\xi_2)=-\xi_1. 
\edm
Since all $\phi_i$ leave the vector space $V:=\mathrm{span}(\xi_1,\xi_2,\xi_3)$ 
invariant and since they are orthonormal, they also leave $V^\perp$ 
invariant. For $X\perp\xi_1,\xi_2,\xi_3,\del_r$ we have
\bdm
J_1(J_2(X))=\phi_1(\phi_2(X))=\phi_3(X)=J_3(X)=-\phi_2(\phi_1(X))=-J_2(J_1(X)).
\edm
For $\xi_1$ we obtain
\begin{align*}
J_1(J_2(\xi_1))=-J_1(\phi_2(\xi_1))=J_1(\xi_3)=-\phi_1(\xi_3)=\xi_2&=J_2(ar\del_r)=-J_2(J_1(\xi_1))
=\phi_3(\xi_1)=J_3(\xi_1)
\end{align*}
and similarly for $\xi_2$, $\xi_3$ and $\del_r$. For  a
 connection as in Theorem (\ref{th:acs}), we have that all 
almost hermitian structures are parallel under $\bar\nabla$ and 
for $X,Y\in TM$
\bdm
[X,Y]=\bar\nabla^{\bar g}_XY-\bar\nabla^{\bar g}_YX
=\bar\nabla_XY-\bar\nabla_YX-\bar T(X,Y).
\edm
Thus the commutator relations are given by
\bea[*]
[\xi_1,\xi_2]& =&  a^2[J_1(\del_r),J_2(\del_r)]
\ =\ a^2(\bar\nabla_{J_1(\del_r)}J_2(\del_r)-\bar\nabla_{J_2(\del_r)}J_1(\del_r))
-\bar T(\xi_1,\xi_2)\\
&=& a^2(J_2(\bar\nabla_{J_1(\del_r)}\del_r)-J_1(\bar\nabla_{J_2(\del_r)}\del_r))
-\bar T(\xi_1,\xi_2)\\
&=& a^2(J_2(J_1(\del_r))-J_1(J_2(\del_r)))-\bar T(\xi_1,\xi_2)\\
&=& -2a^2J_3(\del_r)-\bar T(\xi_1,\xi_2)=2a\xi_3-T(\xi_1,\xi_2).
\eea[*]
The other relations are to be calculated similarly.
\end{proof}
\begin{NB}
 If we rescale the metric such that $a=1$ and if $T=0$, we have 3 K\"ahlerian 
structures on $\bar M$ and thus $3$ Sasakian structures on $M$. Then the 
commutator relations in Theorem \ref{thm:3aks} make sure that the structures 
on $M$ form a $3$-Sasakian structure. This is Lemma $5$ of \cite{B93}: 
A one to one correspondence between  hyperk\"ahler structures on $\bar M$ and 
$3$-Sasaki structures on $M$.
\end{NB}
\begin{NB}
We emphasize that it is not necessary that the three characteristic connections 
$\nabla^{c,i}$, $i=1,2,3$ coincide in order to apply Theorem \ref{thm:3aks},
only the connections $\nabla^i$ with torsion  $T^i=T^{c,i}-2a\eta_i\wedge F_i$ 
have to be equal. If $M$ is a 3-Sasakian manifold,
$T^i=0$ for $i=1,2,3$ and thus $\nabla^1=\nabla^2=\nabla^3=\nabla^g$.
In this case there exists a special $G_2$ structure on $M$ which will be discussed 
in Example \ref{ex:sasakig2}.
\end{NB}
%
%
%
\section{$G_2$ structures -- $\Spin(7)$ structures on the cone}
%
\subsection{Preparations}
%
Let $(M,g,\phi,P)$ be a $G_2$ manifold (see Section \ref{sec:som}).
We cite a classical, but for us crucial result by Fernandez and Gray:
\begin{lem}[{\cite[Lemma 2.7]{FG82}}]\label{lem:starphi}
\begin{align*}
*\phi(V,W,X,Y)&=g(P(V,W),P(X,Y))-g(V,X)g(W,Y)+g(V,Y)g(W,X)\\
&=\phi(V,W,P(X,Y))-g(V,X)g(W,Y)+g(V,Y)g(W,X).
\end{align*}
\end{lem}
\begin{NB}
In \cite{FG82} this  formula is stated differently,
\bdm
*\phi(V,W,X,Y)\ =\ -g(P(V,W),P(X,Y))+g(V,X)g(W,Y)-g(V,Y)g(W,X).
\edm
This is due to the standard $3$-form $\phi$ used by Fern\'andez and Gray,
which  corresponds to the orientation opposite to ours. This changes 
the sign of the Hodge operator. 
\end{NB}
Now we are able to prove
\begin{lem}\label{lem:nablastarphi}
For any metric connection $\nabla$ with skew torsion on $M$, the $G_2$ form
$\phi$ satisfies
\bdm
(\nabla_Z*\phi)(V,W,X,Y)\ =\ (\nabla_Z\phi)(V,W,P(X,Y))+(\nabla_Z\phi)(X,Y,P(V,W)).
\edm
If $\nabla$ satisfies $\nabla\phi=a*\phi$ for some $a>0$, we have the 
simplified relation
\bea[*]
(\nabla_Z*\phi)(V,W,X,Y)& =& a[\phi(X,Y,V)g(Z,W)-\phi(X,Y,W)g(Z,V)\\
&& +\phi(V,W,X)g(Z,Y)-\phi(V,W,Y)g(Z,X)].
\eea[*]
\end{lem}
\begin{proof}
For any metric connection with skew torsion we have
\begin{align*}
 (\nabla_Z*\phi)(V,W,X,Y)=&Z*\phi(V,W,X,Y)-*\phi(\nabla_ZV,W,X,Y)
-*\phi(V,\nabla_ZW,X,Y)\\
&-*\phi(V,W,\nabla_ZX,Y)-*\phi(V,W,X,\nabla_ZY).
\end{align*}
Since $\nabla$ is metric, $g$ is parallel and with Lemma 
\ref{lem:starphi} we get
\begin{align*}
=&Z\phi(V,W,P(X,Y))-\phi(\nabla_ZV,W,P(X,Y))-\phi(V,\nabla_ZW,P(X,Y))-\phi(V,W,P(\nabla_ZX,Y))\\
&-\phi(V,W,P(X,\nabla_ZY))-\phi(V,W,\nabla_ZP(X,Y))+\phi(V,W,\nabla_ZP(X,Y)).
\end{align*}
We have $\phi(V,W,(\nabla_ZP)(X,Y))=g(P(V,W),(\nabla_ZP)(X,Y))
=(\nabla_Z\phi)(X,Y,P(V,W))$ and thus we get
\bdm
(\nabla_Z*\phi)(V,W,X,Y)\ =\ (\nabla_Z\phi)(V,W,P(X,Y))+(\nabla_Z\phi)(X,Y,P(V,W)).
\edm
The condition $\nabla\phi=a*\phi$ implies
\bdm
(\nabla_Z*\phi)(V,W,X,Y)\ =\ -a*\phi(P(X,Y),Z,V,W)-a*\phi(P(V,W),Z,X,Y)
\edm
and aplying once again  Lemma \ref{lem:starphi} yields
\bea[*]
\lefteqn{(\nabla_Z*\phi)(V,W,X,Y)\ = }\\
&=&-a\phi(P(X,Y),Z,P(V,W))-a\phi(P(V,W),Z,P(X,Y))+ag(P(X,Y),V)g(Z,W)\\
&& -ag(P(X,Y),W)g(Z,V)+ag(P(V,W),X)g(Z,Y)-ag(P(V,W),Y)g(Z,X)\\
&=&a[\phi(X,Y,V)g(Z,W)-\phi(X,Y,W)g(Z,V)+\phi(V,W,X)g(Z,Y)-\phi(V,W,Y)g(Z,X)],
\eea[*]
which finishes the proof.
\end{proof}
We define a $4$-form on the cone $\bar M$ via
\bdm
\Phi(\del_r,X,Y,Z)\ :=\ a^3r^3\phi(X,Y,Z),\qquad 
\Phi(X,Y,Z,W)\ :=\ a^4r^4*\phi(X,Y,Z,W)
\edm
for $X,Y,Z,W \in TM$. 
Since $\del_r\lrcorner\Phi$ locally is a $G_2$-structure on 
$\del_r^\perp$, $\Phi$ is a Spin(7)-structure on $\bar M$.
As in Section \ref{sec:acms}, given a characteristic connection on $M$ 
with respect to $\phi$, we construct a connection $\nabla$ with skew 
symmetric torsion $T$ on $M$ such that its lift $\bar \nabla$ to $\bar M$ 
with torsion $\bar T$ is the characteristic connection on $\bar M$ with 
respect to $\Phi$. Since we have $T=\bar T_{|TM}$ and $\del_r\lrcorner \bar
T=0$,  we have $\bar T=T=0$ 
in case of a parallel $\Spin(7)$ structure with respect to the 
Levi-Civita connection on $\bar M$, and thus $\nabla$ 
is the Levi-Civita connection on $M$.
\begin{dfn}\label{def:nabg2}
Let $(M,g,\phi)$ be a $G_2T $ manifold with characteristic  connection
$\nabla^c$. We define a 
metric connection $\nabla$ with skew symmetric torsion $T$ via
\bdm
T:=T^c-\frac{2a}{3}\phi.
\edm 
As in the metric almost contact case (see the comments in Definition 
\ref{def:normaldef}), $T$ cannot be computed abstractly,
but it is found through an educated guess and justified a posteriori from its
properties.
\end{dfn}
\begin{thm}\label{thm:g2str}
The connection $\nabla$ satisfies 
\bdm
\nabla\phi=a*\phi,
\edm
and $\Phi$ is parallel with respect to $\bar \nabla$, the appendant 
connection on $\bar M$. 
\end{thm}
\begin{proof}
 We have for the Riemannian connection $\nabla^g$ on $M$
\bea[*]
 \nabla_X\phi(Y,Z,W)& = & X\phi(Y,Z,W)-\phi(\nabla^g_XY,Z,W)
-\phi(Y,\nabla^g_XZ,W)-\phi(Y,Z,\nabla^g_XW)\\
&&-\frac{1}{2}\phi(T(X,Y),Z,W)-\frac{1}{2}\phi(Y,T(X,Z),W)
-\frac{1}{2}\phi(Y,Z,T(X,W))\\
&=&(\nabla^c_X\phi)(Y,Z,W)+\frac{1}{2}\phi((T^c-T)(X,Y),Z,W)\\
&& +\frac{1}{2}\phi(Y,(T^c-T)(X,Z),W)+\frac{1}{2}\phi(Y,Z,(T^c-T)(X,W))
\eea[*]
and because $\nabla^c\phi=0$ we have
\bea[*]
\lefteqn{\nabla_X\phi(Y,Z,W)\ =}\\
& =&\frac{1}{2}[(T^c-T)(X,Y,P(Z,W))+(T^c-T)(X,Z,P(W,Y))+(T^c-T)(X,W,P(Y,Z))]\\
&=&\frac{a}{3}[\phi(X,Y,P(Z,W))+\phi(X,Z,P(W,Y))+\phi(X,W,P(Y,Z))].
\eea[*]
With Lemma \ref{lem:starphi} we obtain
\bea[*]
 a*\phi(X,Y,Z,W)&=&\frac{a}{3}[*\phi(X,Y,Z,W)+*\phi(X,Z,W,Y)+*\phi(X,W,Y,Z)]\\
&=&\frac{a}{3}[\phi(X,Y,P(Z,W))+\phi(X,Z,P(W,Y))+\phi(X,W,P(Y,Z))\\
&&-g(X,Z)g(Y,W)+g(X,W)g(Y,Z)-g(X,W)g(Z,Y)+g(X,Y)g(Z,W)\\
&&-g(X,Y)g(W,Z)+g(X,Z)g(W,Y)]\\
&=&\nabla_X\phi(Y,Z,W),
\eea[*]
which proves the first statement.
To show $\bar\nabla\Phi=0$ on $\bar M$ we look at several cases. Let 
always be $V,W,X,Y,Z\in TM$.\\
\underline{Case 1:} If $\del_r$ is one of the arguments,
we compute
\begin{align*}
(\bar\nabla_W\Phi)(\del_r,X,Y,Z)\ =\ &Wa^3r^3\phi(X,Y,Z)-\frac{1}{r}\Phi(W,X,Y,Z)-r^3a^3\phi(\nabla_WX,Y,Z)\\
&-r^3a^3\phi(X,\nabla_WY,Z)-r^3a^3\phi(X,Y,\nabla_WZ)
\end{align*}
\bdm
=\ a^3r^3(\nabla_W\phi)(X,Y,Z)-\frac{1}{r}\Phi(W,X,Y,Z)
\ =\ a^4r^3*\phi(W,X,Y,Z)-\frac{1}{r}\Phi(W,X,Y,Z)\ =\ 0.
\edm
\underline{Case 2:} 
If the direction of the derivative is equal to $\del_r$, we obtain
\begin{align*}
(\bar\nabla_{\del_r}\Phi)(X,Y,Z,W)&=\del_r(a^4r^4*\phi(X,Y,Z,W))-4\frac{1}{r}\Phi(X,Y,Z,W)\\
&=\ 4r^3a^4*\phi(X,Y,Z,W)-4\frac{1}{r}\Phi(X,Y,Z,W)\ =\ 0.
\end{align*}
\underline{Case 3:} 
If the direction of the derivative and one argument are equal to $\del_r$ we compute
\bdm
(\bar\nabla_{\del_r}\Phi)(\del_r,X,Y,Z)\ =\ 
\del_r(a^3r^3\phi(X,Y,Z))-3a^3r^3\frac{1}{r}\phi(X,Y,Z)\ =\ 0.
\edm

\underline{Case 4:} On $TM$ we have:
\bea[*]
\lefteqn{(\bar\nabla_V\Phi)(W,X,Y,Z)\ =}\\
&=& a^4r^4V*\phi(W,X,Y,Z)
-\Phi(\bar\nabla_VW,X,Y,Z)-\Phi(W,\bar\nabla_VX,Y,Z)
-\Phi(W,X,\bar\nabla_VY,Z)\\ &&-\Phi(W,X,Y,\bar\nabla_VZ)\\
&=&a^4r^4V*\phi(W,X,Y,Z)-\Phi(\nabla_VW-\frac{1}{r}\bar g(V,W)\del_r,X,Y,Z)
-\Phi(W,\nabla_VX-\frac{1}{r}\bar g(V,X)\del_r,Y,Z)\\
&&-\Phi(W,X,\nabla_VY-\frac{1}{r}\bar g(V,Y)\del_r,Z)-\Phi(W,X,Y,\nabla_VZ
-\frac{1}{r}\bar g(V,Z)\del_r)\\
&=&a^4r^4(\nabla_V*\phi)(W,X,Y,Z)+r^4a^5[g(V,W)\phi(X,Y,Z)-g(V,X)\phi(W,Y,Z)\\
&&+g(V,Y)\phi(W,X,Z)-g(V,Z)\phi(W,X,Y)],
\eea[*]
which is equal to zero due to Lemma \ref{lem:nablastarphi}.
\end{proof}
Conversely, given a $\Spin(7)$ structure $(\bar M,\bar g,\Phi,\bar P,\bar p)$ 
on $\bar M$ (see Section \ref{sec:som} for the definitions), 
$\del_r\lrcorner\Phi$ is a $G_2$ structure with respect to the metric 
$a^2g$ on $M=M\times\{1\}\subset \bar M$ and thus
\bdm
\phi\ :=\ \frac{1}{a^3}\del_r\lrcorner\Phi
\edm
defines a $G_2$ structure on $M$ with respect to the metric $g$. 
To prove the following theorem, we need
\begin{lem}\label{lem:hodge}
 If $*$ is the Hodge operator on $M$ with respect to $g$ and $*_{a^2g}$ 
is the Hodge operator on $M$ with respect to the metric $a^2g$, we have 
for any $3$-form $\omega$
\bdm
*_{a^2g}\omega=a*\omega.
\edm
\end{lem}
\begin{proof}
 Let $e_i$ for $i=1..7$ be an orthonormal basis with dual basis $e^i$ on 
$M$ with respect to $g$. Then $\frac{1}{a}e_i$ with dual $ae^i$ is a 
orthonormal basis with respect to $a^2g$. We define 
$e^{\{i,j,k\}}:=e^i\wedge e^j\wedge e^k$ and 
$e^{\{i,j,k,j\}}:=e^i\wedge e^j\wedge e^k\wedge e^l$ as well as 
$(se)^{\{i,j,k\}}:=se^i\wedge se^j\wedge se^k$ for $s\in\R$ and
$(se)^{\{i,j,k,j\}}$ respectively. Then we have
\bdm
*_{a^2g}e^{\{i,j,k\}}=\frac{1}{a^3}*_{a^2g}(ae)^{\{i,j,k\}}=\frac{1}{a^3}(ae)^{\{1,..,7\}\backslash\{i,j,k\}}=\frac{1}{a^3}a^4e^{\{1,..,7\}\backslash\{i,j,k\}}=a*e^{\{i,j,k\}},
\edm
which proves the lemma.
\end{proof}
\begin{thm}
 Given a $\Spin(7)$ structure on $\bar M$ with characteristic connection 
$\bar\nabla$ being the lift of a connection $\nabla$ on $M$, we have for 
the $G_2$ structure $\phi$ induced by $\Phi$
\bdm
\nabla\phi=a*\phi
\edm
and the characteristic connection on $(M,g,\phi)$ is given by 
$T^c=T+\frac{2a}{3}\phi$.
\end{thm}
\begin{proof}
 We have for $W,X,Y,Z\in TM$
\begin{align*}
(\nabla_W\phi)(X,Y,Z)\ =&\ \frac{1}{a^3}[W\Phi(\del_r,X,Y,Z)\\
&-\Phi(\del_r,\nabla_WX,Y,Z)-\Phi(\del_r,X,\nabla_WY,Z)-\Phi(\del_r,X,Y,\nabla_WZ)]\\
\ =&\ \frac{1}{a^3}[(\bar\nabla_W\Phi)(\del_r,X,Y,Z)+\Phi(\bar\nabla_W\del_r,X,Y,Z)]=\frac{1}{a^3}\Phi(W,X,Y,Z).
\end{align*}
With Lemma $8$ of \cite{B93} and the definition of $\phi$ we conclude 
$\Phi|_{TM}=*_{a^2g}(\del_r\lrcorner\Phi)=*_{a^2g}(a^3\phi)=a^4*\phi$, 
where $*_{a^2g}$ is the Hodge operator on $M\subset\bar M$ with respect 
to the metric $a^2g$. The last equality follows from Lemma \ref{lem:hodge}. 
Thus we get
\bdm
\nabla\phi=a*\phi.
\edm
For the connection $\nabla^c$ with torsion $T^c=T+\frac{2a}{3}\phi$ we 
calculate as in the proof of Theorem \ref{thm:g2str}
\begin{align*}
 (\nabla^c_X\phi)(Y,Z,W)\ =&\ (\nabla_X\phi)(Y,Z,W)+\frac{1}{2}[(T-T^c)(X,Y,P(Z,W))
+(T-T^c)(X,Z,P(W,Y))\\
&+(T-T^c)(X,W,P(Y,Z))]\\
\ =&\ a*\phi(X,Y,Z,W)-\frac{a}{3}[\phi(X,Y,P(Z,W))+\phi(X,Z,P(W,Y))+\phi(X,W,P(Y,Z))]
\end{align*}
which is equal to zero due to Lemma \ref{lem:starphi}. Since the
characteristic 
connection of a $G_2$ manifold is unique, this proves the Theorem.
\end{proof}
\begin{NB}
As in the metric almost contact case, $T= T^c - \frac{2a}{3}\phi$ measures
the `deviation' of the $G_2$ structure from a nearly parallel $G_2$ structure;
for then, $ T^c = \frac{2a}{3}\phi$, i.\,e.~$T=0$ and thus $\nabla=\nabla^g$
lifts to the Levi-Cita connection on $\bar M$, reflecting the fact
that the $\Spin(7)$ structure on the cone is then integrable. That $\nabla$
plays indeed a geometric role beyond being an auxiliary tool, and that
this role is that of a the Levi-Civita connection
for a nearly parallel $G_2$ manifold,  is confirmed by Theorem
\ref{thm:g2str}, since it states that the equation $\nabla^g\phi=a*\phi$
for the nearly parallel case generalizes to $\nabla\phi=a*\phi$ for
any $G_2 T$ manifold.

\end{NB}
%
\subsection{The classification of $G_2$ structures and the corresponding 
classification of $\Spin(7)$ structures on the cone}
%
We will now discuss the classification of Fern\'andez \cite{F85} of 
$\Spin(7)$ structures on $\bar M$ given in Section \ref{sec:som}, and
compute the correspondence to the classification of $G_2$ structures 
\cite{FG82}. Again we are only interested in structures carrying a 
characteristic connection ($G_2$ structures of class 
$\mathcal W_1\oplus\mathcal W_3\oplus\mathcal W_4$).
We write $X_M$ for the projection on $TM$ of a vector field $X$ in $T\bar M$.
We summarize some useful identities:
\begin{lem}\label{lam:barP} 
\begin{enumerate}
\item[]
\item $P$ can be expressed through $\phi$ on $TM$:\ $P(Y,Z) \, =\, \sum_l\phi(e_l,Y,Z)e_l$.
\item
For any metric connection $\tilde\nabla$ with skew torsion on $M$, we have:
\bea[*]
(\tilde\nabla_X\phi)(Y,Z,V)&=& g((\tilde\nabla_XP)(Y,Z),V),\\
(\tilde\nabla_XP)(Y,Z) & =& \sum_lg(e_l,(\tilde\nabla_XP)(Y,Z))e_l
=\sum_l(\tilde\nabla_X\phi)(e_l,Y,Z)e_l.
\eea[*]
\item For $\nabla$, this can be simplified to \ 
$(\nabla_XP)(Y,Z)\, =\, a\sum_l*\phi(X,e_l,Y,Z)e_l$.
\item $P,\, \phi$, and $\bar P$ are related by $(X,Y,Z\in TM)$
\bdm
\bar g(\bar P(X,Y,Z),\del_r)
=-a^3r^3\phi(X,Y,Z),\quad
\bar P(\del_r,X,Y)=arP(X,Y),\quad\bar P(Y,Z,V)_M=ar^2(\nabla_YP)(Z,V).
\edm
\item The derivative of $\Phi$ on $\bar M$ can be expressed in 
terms of $\phi$ on $M$ $(X,Y,Z,V,W\in TM)$:
\bdm
(\bar\nabla^{\bar g}_X\Phi)(\del_r,Z,V,W)=a^3r^3[(\nabla^g-\nabla)_X\phi](Z,V,W),
(\bar\nabla^{\bar g}_X\Phi)(Y,Z,V,W)=a^4r^4[(\nabla^g-\nabla)_X*\phi](Y,Z,V,W).
\edm
\end{enumerate}
\end{lem}
\begin{proof}
Statements (1)-(3) are easily checked. To prove statement (4) for $X,Y,Z\in TM$,
we have
\bdm
\bar g(\bar P(\del_r,X,Y),Z)=\Phi(\del_r,X,Y,Z)=a^3r^3\phi(X,Y,Z)=ar\bar g(P(X,Y),Z),
\edm
thus  $\bar P(\del_r,X,Y)=arP(X,Y)$. Furthermore,
\bea[*]
 \bar g(X,\bar P(Y,Z,V))&=&\Phi(Y,Z,V,X)\ =\ a^3r^4(\nabla_Y\phi)(Z,V,X)
\ =\ a^3r^4g(X,(\nabla_YP)(Z,V))\\
&=&ar^2\bar g(X,(\nabla_YP)(Z,V)),
\eea[*]
and thus $\bar P(Y,Y,V)_M=ar^2(\nabla_YP)(Z,V)$. For (5) and
vector fields  $X,Y,Z,V,W\in TM$, we calculate 
\bea[*]
\lefteqn{2(\bar\nabla^{\bar g}_X\Phi)(\del_r,Z,V,W)\ =}\\
&=&2(\bar\nabla_X\Phi)(\del_r,Z,V,W)+\Phi(\del_r,\bar T(X,Z),V,W)
+\Phi(\del_r,Z,\bar T(X,V),W)+\Phi(\del_r,Z,V,\bar T(X,W))\\
&=&a^3r^3[\phi(T(X,Z),V,W)+\phi(Z,T(X,V),W)+\phi(Z,V,T(X,W))]\\
&=&2a^3r^3[\phi((\nabla_X-\nabla^g_X)Z,V,W)
+\phi(Z,(\nabla_X-\nabla^g_X)V,W)+\phi(Z,V,(\nabla_X-\nabla^g_X)W)]\\
&=&-2a^3r^3[(\nabla-\nabla^g)_X\phi](Z,V,W),
\eea[*]
and similarly
\bea[*]
\lefteqn{(\bar\nabla^{\bar g}_X\Phi)(Y,Z,V,W)\ =}\\
&=&\frac{1}{2}[\Phi(\bar T(X,Y),Z,V,W)+\Phi(Y,\bar T(X,Z),V,W)
+\Phi(Y,Z,\bar T(X,V),W)+\Phi(Y,Z,V,\bar T(X,W))]\\
&=&\frac{a^4r^4}{2}[*\phi(T(X,Y),Z,V,W)+*\phi(Y,T(X,Z),V,W)
+*\phi(Y,Z,T(X,V),W)+*\phi(Y,Z,V,T(X,W))]\\
&=&-a^4r^4[(\nabla-\nabla^g)_X*\phi](Y,Z,V,W)
=a^4r^4[(\nabla^g-\nabla)_X*\phi](Y,Z,V,W),
\eea[*]
which finishes the proof.
\end{proof}
\begin{NB}\label{rem:eindeutnabla}
Since the  characteristic connection of the $\Spin(7)$  structure on $\bar M$
is unique (see Section \ref{sec:som}), we can conclude for
any such structure satisfying $\bar\nabla^{\bar g}\Phi=0$ that $\nabla=\nabla^g$ and thus 
$\nabla^g\phi=a*\phi$ and the $G_2$ structure is of class $\mathcal W_1$.
 Conversely, given a connection $\nabla$ with skew symmetric torsion and 
$\nabla\phi=a*\phi$ we construct $\nabla^c$ via $T^c:=T-\frac{2a}{3}\phi$, 
which satisfies $\nabla^c\phi=0$ and thus is unique. Hence a metric 
connection with skew symmetric torsion and the property $\nabla\phi=*\phi$ is unique.
\end{NB}
\begin{dfn}
For any tensor $R$ on $M$ let 
\bdm
R\llcorner X\ :=\ R(-,X).
\edm
We extend the metric $g$ to arbitrary $k$-tensors $R,S$ via an orthonormal
frame $e_1,\ldots,e_n$
\bdm
g(R,S)\ :=\ \sum_{i_1,..,i_k=1}^nR(e_{i_1},..,e_{i_k})S(e_{i_1},..,e_{i_k}).
\edm 
\end{dfn}
\begin{lem}\label{lem:spin7str}
A $\Spin(7)$ structure on $\bar M$ is of class $\mathcal U_1$ if and 
only if on $M$
\begin{itemize} 
\item $g(\nabla^g\phi,*\phi)=ag(*\phi,*\phi)$, and 
\item for every $X\in TM$ we have 
$g(*\phi,[(\nabla-\nabla^g)*\phi]\llcorner X)
=3g(\phi,[(\nabla-\nabla^g)\phi]\llcorner X)$.
\end{itemize}
The structure on $\bar M$ is of class $\mathcal U_2$ if and only if 
the following conditions are satisfied for $X,Y,Z,X_1,..,X_4\in TM$ and a local 
orthonormal frame $e_1,..,e_7$ of $TM$:
\begin{itemize}
\item $\delta \Phi|_{TM}=0$ on $TM$, which is equivalent to 
$0=\sum\limits_{i=1}^7[(\nabla^g-\nabla)_{e_i}*\phi](e_i,X,Y,Z)$
\item $0=\sum\limits^4_{i=1}\sum\limits_{l<j<8}(-1)^i\delta\phi(e_l,e_j)
\phi(e_l,e_j,X_i)\phi(X_1,..,\hat X_i,..,X_4)$
\item 
$28[(\nabla^g-\nabla)_W*\phi](X_1,X_2,X_3,X_4)
=\sum\limits^4_{i=1}\sum\limits_{l<j<8}(-1)^{i+1}\delta\phi(e_l,e_j)\phi(e_l,e_j,X_i)*\phi(W,X_1,..,\hat X_i,..,X_4)$.
\end{itemize}
\end{lem}
\begin{proof}
We consider a local $\bar g$-orthonormal frame
$\bar e_1=\frac{1}{ar}e_1,..,\bar e_7=\frac{1}{ar}e_7,e_8=\del_r$ of 
$T\bar M$ such that $e_1,..,e_7$ is a local orthonormal frame of $TM$. With Lemma 4.2 
of \cite{F85} a $\Spin(7)$ structure is defined to be of class $\mathcal U_1$ 
if and only if
\bdm
0\ =\ -6\delta\Phi(\bar p(X))
\ =\ \sum^8_{i,k,j=1}(\bar\nabla^{\bar g}_{\bar e_i}\Phi)(\bar e_j,\bar e_k,
\bar P(\bar e_i,\bar e_j,\bar e_k),X).
\edm
For $X\in TM$ we have
\bea[*]
0&=& -6\delta\Phi(\bar p(X))\ =\ \sum^8_{i,k,j=1}(\bar\nabla^{\bar g}_{\bar
  e_i}\Phi)
(\bar e_j,\bar e_k,\bar P(\bar e_i,\bar e_j,\bar e_k),X)\\
& =& \sum^7_{i,k,j=1}(\bar\nabla^{\bar g}_{\bar e_i}\Phi)(\bar e_j,\bar e_k,\bar
P(\bar e_i,\bar e_j,\bar e_k),X)+2\sum^7_{i,j=1}(\bar\nabla^{\bar g}_{\bar
  e_i}\Phi)(\bar e_j,\del_r,\bar P(\bar e_i,\bar e_j,\del_r),X)\\
&=&\frac{1}{a^6r^6}\sum^7_{i,k,j=1}(\bar\nabla^{\bar g}_{e_i}\Phi)
(e_j,e_k,ar^2(\nabla_{e_i}P)(e_j,e_k)+\bar g(\bar P(e_i,e_j,e_k),\del_r)\del_r,X)\\
&&+2\frac{1}{a^4r^4}\sum^7_{i,j=1}(\bar\nabla^{\bar g}_{e_i}\Phi)
(e_j,\del_r,arP(e_i, e_j),X)\\
&=&\frac{1}{a^5r^4}\sum^7_{i,k,j=1}a^4r^4[(\nabla^g-\nabla)_{e_i}*\phi]
(e_j,e_k,(\nabla_{e_i}P)(e_j,e_k),X)\\
&&-\frac{1}{a^3r^3}\sum^7_{i,k,j=1}\phi(e_i,e_j,e_k)
(\bar\nabla^{\bar g}_{e_i}\Phi)(e_j,e_k, \del_r,X)
-2\frac{a^3r^3}{a^3r^3}\sum^7_{i,j=1}([\nabla^{g}-\nabla]_{e_i}\phi)(e_j,P(e_i, e_j),X)\\
&=&\sum^7_{i,k,j,l=1}[(\nabla^g-\nabla)_{e_i}*\phi](e_j,e_k,*\phi(e_i,e_l,e_j,e_k)e_l,X)
-3\sum^7_{i,k,j=1}\phi(e_i,e_j,e_k)([\nabla^{g}-\nabla]_{e_i}\phi)(e_j,e_k,X)\\
&=&g(*\phi,(\nabla^g-\nabla)*\phi\llcorner X)-3g(\phi,(\nabla^g-\nabla)\phi\llcorner X).
\eea[*]
In case $X=\del_r$, we deduce from Lemma \ref{lam:barP}:
\begin{align*}
 0=&\sum^7_{i,j,k=1}(\bar\nabla^{\bar g}_{e_i}\Phi)
(e_j,e_k,\bar P(e_i,e_j,e_k),\del_r) \ =\  ar^2\sum^7_{i,j,k=1}
(\bar\nabla^{\bar g}_{e_i}\Phi)(e_j,e_k,(\nabla_{e_i} P)(e_j,e_k),\del_r)\\
 =&-a^4r^5\sum^7_{i,j,k=1}[(\nabla^{g}-\nabla)_{e_i}\phi]
(e_j,e_k,(\nabla_{e_i} P)(e_j,e_k))\\
=&-a^4r^5[\sum^7_{i,j,k,l=1}(\nabla^{g}_{e_i}\phi)(e_j,e_k,e_l)
(\nabla_{e_i} \phi)(e_j,e_k,e_l)-\sum^7_{i,j,k,l=1}(\nabla_{e_i}\phi)
(e_j,e_k,e_l)(\nabla_{e_i} \phi)(e_j,e_k,e_l)]\\
=&-a^4r^5[g(\nabla^g\phi,\nabla\phi)-g(\nabla\phi,\nabla\phi)]
=-a^5r^5[g(\nabla^g\phi,*\phi)-ag(*\phi,*\phi)],
\end{align*}
and thus we have $g(\nabla^g\phi,*\phi)=ag(*\phi,*\phi)$.
A $\Spin(7)$ structure is of class $\mathcal U_2$ if it satisfies
\be\label{eq:caseW2}
\begin{split}
28(\bar\nabla^{\bar g}_W\Phi)(X_1,X_2,X_3,X_4)=-\sum^4_{i=1}(-1)^{i+1}[&\delta\Phi(\bar p(X_i))\Phi(W,X_1,..,\hat X_i,..,X_4)\\
&+7\bar g(W,X_i)\delta\Phi(X_1,..,\hat X_i,..,X_4)].
\end{split}
\ee
Suppose $W=X_1=\del_r$ and $X_2,X_3,X_4\in TM$. For a $3$-form $\xi$ on $TM$ we have
\bdm
\bar g(\bar p(\del_r),\xi)=\bar g(\del_r,\bar P(\xi))=-\Phi(\del_r,\xi)
=-a^3r^3\phi(\xi)=\bar g(-a^3r^3\phi,\xi)
\edm
and thus $\bar p(\del_r)=-a^3r^3\phi$. Since $\del_r\lrcorner \bar T=0$ 
we have $\bar \nabla^{\bar g}_{\del_r}\Phi=0$ and the defining relation of 
the class $\mathcal U_2$ reduces to
\bdm
0=\delta\Phi(p(\del_r))\Phi(\del_r,X_2,X_3,X_4)+7\delta\Phi(X_2,X_3,X_4)
=\delta\Phi(-a^6r^6\phi(X_2,X_3,X_4)\phi+7X_2\wedge X_3\wedge X_4).
\edm
Since $a^6r^6\phi(X_2,X_3,X_4)\phi-7X_2\wedge X_3\wedge X_4$ spans 
$\Lambda^3(TM)$ we have $\delta \Phi=0$ on $TM$. 
For $X,Y,Z\in TM$ we have
\begin{align*}
0=\delta \Phi(X,Y,Z)=& -\sum_{i=1}^8(\bar\nabla^{\bar g}_{\bar e_i}\Phi)(\bar e_i,X,Y,Z)
=-\frac{1}{a^2r^2}\sum_{i=1}^7(\bar\nabla^{\bar g}_{e_i}\Phi)(e_i,X,Y,Z)\\
=&-a^2r^2\sum_{i=1}^7[(\nabla^g-\nabla)_{e_i}*\phi](e_i,X,Y,Z).
\end{align*}
For $X\in TM$ we have
\begin{align*}
\delta\Phi(\bar p(X))=
&\delta\Phi(\sum_{i<j<k=1}^8\bar 
g(\bar p(X),\bar e_i\wedge\bar e_j\wedge\bar e_k)\bar e_i\wedge\bar
e_j\wedge\bar e_k)
=\delta\Phi(\sum_{i<j<8}\bar g(\bar p(X),\bar e_i\wedge\bar e_j\wedge\bar e_8)
\bar e_i\wedge\bar e_j\wedge\bar e_8)\\
=&\sum_{i<j<8}\bar g(\bar p(X),\bar e_i\wedge\bar e_j\wedge\bar e_8)
\delta\Phi(\bar e_i,\bar e_j,\del_r)
=-\sum_{k=1}^7\sum_{i<j<8}(\bar\nabla^{\bar g}_{\bar e_k}\Phi)
(\bar e_k,\bar e_i,\bar e_j,\del_r)\bar g(X,\bar P(\bar e_i,\bar e_j,\del_r))\\
=&\sum_{k=1}^7\sum_{i<j<8}a^3r^3(\nabla^g_{\bar e_k}\phi)
(\bar e_k,\bar e_i,\bar e_j)\Phi(\bar e_i,\bar e_j,\del_r, X)
=a^6r^6\sum_{k=1}^7\sum_{i<j<8}(\nabla^g_{\bar e_k}\phi)
(\bar e_k,\bar e_i,\bar e_j)\phi(\bar e_i,\bar e_j, X)\\
=&-\sum_{i<j<8}\delta\phi(e_i,e_j)\phi(e_i,e_j,X).
\end{align*}
Suppose $W=\del_r$ and $X_1,..,X_4\in TM$. Then equation (\ref{eq:caseW2}) gives us
\bea[*]
0& =& \sum^4_{i=1}(-1)^{i+1}\delta\Phi(\bar p(X_i))a^3r^3\phi(X_1,..,\hat X_i,..,X_4)\\
&=& a^3r^3\sum^4_{i=1}\sum_{l<j<8}(-1)^i\delta\phi(e_l,e_j)\phi(e_l,e_j,X_i)
\phi(X_1,..,\hat X_i,..,X_4).
\eea[*]
For $W,X_i\in TM$, equation (\ref{eq:caseW2}) reduces to
\begin{align*}
28&(\bar\nabla^{\bar g}_W\Phi)(X_1,X_2,X_3,X_4)
=28a^4r^4[(\nabla^g-\nabla)_W*\phi](X_1,X_2,X_3,X_4),
\end{align*}
which is equal to
\begin{align*}
&=-\sum^4_{i=1}(-1)^{i+1}[\delta\Phi(\bar p(X_i))
\Phi(W,X_1,..,\hat X_i,..,X_4)+7\bar g(W,X_i)\delta\Phi(X_1,..,\hat X_i,..,X_4)]\\
=&a^4r^4\sum^4_{i=1}\sum_{l<j<8}(-1)^{i+1}\delta\phi(e_l,e_j)
\phi(e_l,e_j,X_i)*\phi(W,X_1,..,\hat X_i,..,X_4).
\end{align*}
This proves the statement.
%
%
%
\end{proof}
\begin{NB}
One can use Lemma \ref{lem:nablastarphi} and Lemma \ref{lam:barP} to 
simplify these equations in rather lengthly calculations. The property 
\bdm
0=\sum_{i=1}^7[(\nabla^g-\nabla)_{e_i}*\phi](e_i,X,Y,Z)
\edm
can for example be simplified to
\bdm
0=g((\phi\llcorner Y)\llcorner Z,\delta \phi\llcorner X)
+g(\phi\llcorner X,(\nabla^g\phi\llcorner Y)\llcorner Z)
-g(\phi\llcorner X,(*\phi\llcorner Y)\llcorner Z).
\edm
Another simplification (see Lemma \ref{lem:seccondlem}) will be 
used in the example.
\end{NB}
\begin{thm}\label{thm.whenU1}
If the  $\Spin(7)$ structure 
on the cone $\bar M$ is of class $\mathcal U_1$, then:
\begin{itemize}
 \item The $G_2$ structure $\phi$ on $M$ cannot be of class 
$\mathcal W_3\oplus \mathcal W_4$.
 \item The $G_2$ structure is of class $\mathcal W_1$ if and only if the $\Spin(7)$ structure is integrable.
\end{itemize}
If the structure on $\bar M$ is of class $\mathcal U_2$, then the structure on $M$ is never of class $\chi_1\oplus\chi_3$.
\end{thm}
\begin{proof}
Since the relation $g(\nabla^g\phi,*\phi)=0$ defines the class 
$\mathcal W_2\oplus \mathcal W_3\oplus \mathcal W_4$, we
conclude the first result  directly  from Lemma \ref{lem:spin7str}. Now,
assume the $G_2$ structure $\phi$ is of class $\mathcal W_1$, i.\,e.~nearly
parallel $G_2$ (see \cite{FG82}):
\bdm
\nabla^g\phi=\frac{1}{168}g(\nabla^g\phi,*\phi)*\phi.
\edm
Taking the scalar product with $*\phi$ on both sides leads to
\bdm
g(\nabla^g\phi,*\phi)=\frac{1}{168}g(\nabla^g\phi,*\phi)g(*\phi,*\phi).
\edm
With the $\Spin(7)$ structure being of class $\mathcal U_1$ and the 
calculation above we get $g(*\phi,*\phi)
=\frac{1}{168}g(*\phi,*\phi)g(*\phi,*\phi)$ and thus $g(*\phi,*\phi)=168$. 
Therefore,
\bdm
\nabla^g\phi=\frac{1}{168}g(\nabla^g\phi,*\phi)*\phi
=a\frac{1}{168}g(*\phi,*\phi)*\phi=a*\phi.
\edm
Thus $\nabla^g\phi=\nabla\phi=a*\phi$ and with Remark 
\ref{rem:eindeutnabla} we get $\nabla=\nabla^g$ and 
$\bar\nabla^{\bar g}=\bar \nabla$. Since $\bar\nabla\Phi=0$ the 
$\Spin(7)$ structure on $\bar M$ is integrable.\\
Consider a structure on $\bar M$ of class $\mathcal{U}_2$. With Lemma 
\ref{lem:spin7str} we get $\delta \Phi=0$ on $TM$. To see that this 
structure is integrable it is sufficient to show 
$\del_r\lrcorner\delta \Phi=0$, see \cite{F85}. We have for $X,Y\in TM$ 
\begin{align*}
(\del_r\lrcorner\delta \Phi)(X,Y)=&-\sum_{i=1}^8(\bar\nabla^{\bar g}_{\bar e_i}\Phi)(\bar e_i,\del_r,X,Y)=ar\sum_{i=1}^7((\nabla^g-\nabla)_{e_i}\phi)(e_i,X,Y)
=-ar\delta \phi(X,Y).
\end{align*}
This is equal to zero if the structure on $M$ is cocalibrated (of 
class $\chi_1\oplus \chi_3$, defined by $\delta\phi=0$).
\end{proof}
%
\subsection{Corresponding spinors on $G_2$ manifolds and their cones}
%
Since we have $T-T^c=-\frac{2a}{3}\phi$, the difference $\bar T-\overline{T^c}$ 
is the lift of $a^2r^2T-a^2r^2T^c=-\frac{2a}{3}a^2r^2\phi$. Furthermore, 
$\frac{1}{a^3r^3}\del_r\lrcorner\Phi$ is the lift of $\phi$ to $\bar M$, hence 
we have 
\bdm
\bar T-\overline{T^c}=-\frac{2}{3r}\del_r\lrcorner\Phi.
\edm
Now Lemma \ref{lem:corrksps} implies:
\begin{thm}\label{thm:g2killingspinors}
 For a $G_2T $ manifold with characteristic connection $\nabla^c$ and for 
$\alpha=\frac{1}{2}a$ or $\alpha=-\frac{1}{2}a$, there is
\begin{enumerate}
 \item a one to one correspondence between Killing spinors with torsion
\bdm
\nabla^s_X\psi=\alpha X\psi
\edm
on $M$, and parallel 
spinors of the connection $\bar\nabla^s+\frac{4s}{3r}\del_r\lrcorner\Phi$ on
$\bar M$ with cone constant $a$ 
\bdm
\bar\nabla^s_X\psi+\frac{2s}{3r}(X\lrcorner(\del_r\lrcorner\Phi))\psi=0.
\edm
\item a one to one correspondence between $\bar\nabla^s$-parallel 
spinors on $\bar M$ with cone constant $a$ and spinors on $M$ satisfying
\bdm
\nabla^s_X\psi=\alpha X\psi+\frac{2as}{3}(X\lrcorner\phi)\psi.
\edm
\end{enumerate}
In particular for $s=\frac{1}{4}$ we get the correspondence
\begin{center}
\begin{tabular}{|c|c|}
 \hline
spinors on $M$ & spinors on $\bar M$\\
\hline
$\nabla^c_X\psi=\alpha X\psi$ & $\bar\nabla_X\psi
=-\frac{1}{6r}(X\lrcorner(\del_r\lrcorner\Phi))\psi$\\
\hline
$\nabla^c_X\psi=\alpha X\psi+\frac{a}{6}(X\lrcorner \phi)~\psi$ & $\bar\nabla_X\psi=0$\\
\hline
\end{tabular}
\end{center}
\end{thm}
\begin{NB}
As for metric almost contact structures (see Remark \ref{rem:chartorausschr}),
 one can use the characterisation 
$\bar T=-\delta\Phi-\frac{7}{6}*(\theta\wedge\Phi)$ with 
$\theta=\frac{1}{7}*(\delta\Phi\wedge\Phi)$ (see \cite{I04}) and the 
description of $T^c$ given in Theorem 4.8 of \cite{FI02} to rewrite these
equations in terms of the geometric data of the $\Spin(7)$ structure.
\end{NB}
Theorem \ref{thm:g2killingspinors} states, as before, the general
correspondence between spinors on the base and spinors on the cone.
However, $G_2T$ manifolds, i.\,e.~carrying a  characteristic connection $\nabla^c$, 
 enjoy a further, very special property:
The $G_2$ structure $\phi$ induces a unique spinor field $\psi$
of length one and this spinor field is $\nabla^c$-parallel, $\nabla^c \psi=0$.
This is due to the fact that $G_2$ is the stabilizer of a generic spinor in 
$\Delta_7$, the spin representation in dimension $7$.
For a nearly parallel $G_2$ manifold, it is well-known that $\psi$ is just
the Riemannian Killing spinor (see  \cite{FI02}, \cite{Friedrich&I3},
\cite{FK}, \cite{FKMS} for all these results). Thus, $\psi$ induces in this
case the $\nabla^g$-parallel spinor of the integrable $\Spin(7)$ structure on
the cone. We prove that this result carries over to all admissible
$G_2$ manifolds.
\begin{cor}\label{cor:g2explkilling}
 Let $(M,g,\phi)$ be a $G_2T$ manifold with  characteristic connection 
$\nabla^c$, $\psi$ the $\nabla^c$-parallel spinor field defined by $\phi$.
Then $\psi$ satisfies 
\bdm
\nabla^c_X\psi \ =\ - \frac{a}{2}X\psi +\frac{a}{6}(X\lrcorner\phi)\psi
\edm
for every  $a>0$ and induces a $\bar\nabla$-parallel spinor on the cone $\bar
M$, constructed with cone constant $a$ and endowed with its induced 
$\Spin(7)$ structure.
\end{cor}
\begin{proof}
The crucial observation is  the algebraic identity
\bdm
(X\lrcorner\phi)\cdot\psi\ =\ 3\, X\cdot \psi 
\edm
that holds for all vector fields $X$. Since The $7$-dimensional standard
representation $\R^7$ of $G_2$ is isomorphic to the $G_2$ representation
\bdm
\Lambda^2_7 \ =\ \{X\lrcorner \phi\, |\, X\in\R^7\}\, \subset\,
\Lambda^2(\R^7)
\ =\ \so(7) = \Lambda^2_7 \oplus \g_2,
\edm
it is clear that there exists a constant $c$ 
s.\,t.~$(X\lrcorner\phi)\cdot\psi = c X\cdot \psi$; 
one then computes its explicit value in
any realization of the spin representation. Thus, the equation for $\psi$
follows and we can apply  Theorem \ref{thm:g2killingspinors}.
\end{proof}
Be cautious that $\nabla^c$ may have more parallel spinor fields than
just $\psi$; for these, we cannot define a suitable `lifted' spinor on the
cone, unless one finds a similar trick to write the spinor field equation
in a form covered by Theorem \ref{thm:g2killingspinors}.
\begin{NB}
In Theorem 1.1 of \cite{I04} S.~Ivanov proves that any $\Spin(7)$ manifold
admits a spinor field that is parallel with respect to the characteristic 
connection. Corollary \ref{cor:g2explkilling} gives an explicit construction
of this spinor in case the $\Spin(7)$ manifold is the cone of an admissible 
$G_2$ manifold.
\end{NB}
\begin{NB}
Since Corollary \ref{cor:g2explkilling} holds for any $G_2T$ manifold,
one could also carry out the whole study without  using
the $3$-form $\phi$ and the $4$-form $\Phi$: the spinor field $\psi$
describes the $G_2$ structure completely, then one considers the 
induced $\bar\nabla$-parallel spinor $\varphi$ on the cone described in  
Corollary \ref{cor:g2explkilling} and establishes the correspondence
between the $G_2$ classes and the $\Spin(7)$ classes by studying the
equations satisfied by $\psi$ and $\varphi$.
\end{NB}
%
\subsection{Examples}
%
To simplify the calculations in the example we reformulate the second 
condition for a $G_2$ structure on $M$ to imply a $\Spin(7)$ structure 
of class $\mathcal U_1$ on $\bar M$ of Lemma \ref{lem:spin7str}. So we only 
have to calculate $\phi$, $*\phi$ and $\nabla^g\phi$ to check the conditions.
We omit the proof of the following result, it is a lengthy, but
straight forward continuation of the calculations in the proof of 
Lemma \ref{lem:spin7str} and Lemma \ref{lem:nablastarphi}.
\begin{lem}\label{lem:seccondlem}
The second condition of Lemma $\ref{lem:spin7str}$  
\bdm
g(*\phi,[(\nabla-\nabla^g)*\phi]\llcorner X) \ =\ 
3g(\phi,[(\nabla-\nabla^g)\phi]\llcorner X)
\edm
is equivalent to

\begin{align*}
0=&\!\!  \sum^7_{i,k,j,l,m=1} 
\bigg[*\phi(e_i,e_j,e_k,e_l)(\nabla^g_{e_i}\phi)(e_j,e_k,e_m)
\phi(e_m,e_l,X)+
*\phi(e_i,e_j,e_k,e_l)(\nabla^g_{e_i}\phi)(e_l,X,e_m)
\phi(e_m,e_j,e_k) \\
&-*\phi(e_i,e_j,e_k,e_l)*\phi(e_i,e_j,e_k,e_m)\phi(e_m,e_l,X)
-*\phi(e_i,e_l,e_j,e_k)*\phi(e_i,e_l,X,e_m)\phi(e_m,e_j,e_k)\bigg]\\
&+3\sum^7_{i,k,j=1}\bigg[ -\phi(e_i,e_j,e_k)(\nabla^{g}_{e_i}\phi)(e_j,e_k,X)
+a\, \phi(e_i,e_j,e_k)*\phi(e_i,e_j,e_k,X)\bigg].
\end{align*}
\end{lem}
\begin{exa}\label{ex:sasakig2}
 Let $(M,\xi_1,\xi_2,\xi_3,\eta_1,\eta_2,\eta_3)$ be a $7$ dimensional 
$3$-Sasaki manifold with corresponding $2$-forms $F_i$, $i=1,2,3$. Let 
$\eta_i$ for $i=1,..,7$ be the dual of a local basis 
$\{e_1=\xi_1,e_2=\xi_2,e_3=\xi_3,e_4,..,e_7\}$, such that 
\bdm
F_1=-\eta_{23}-\eta_{45}-\eta_{67},~~~F_2=\eta_{13}-\eta_{46}
+\eta_{57},~~~F_3=-\eta_{13}-\eta_{47}-\eta_{56}.
\edm
Here for $\eta_i\wedge..\wedge\eta_j$ we write $\eta_{i,..,j}$.  
In \cite{AF10} it is explained  that there is 
no characteristic connection as such, but one can construct a cocalibrated
$G_2$ structure 
\bdm
\phi=\eta_1\wedge F_1+\eta_2\wedge F_2+\eta_3\wedge F_3
+4\eta_1\wedge\eta_2\wedge\eta_3=\eta_{123}-\eta_{145}-\eta_{167}
-\eta_{246}+\eta_{257}-\eta_{347}-\eta_{356}
\edm
with characteristic connection $\nabla^c$ and torsion 
$T^c=\eta_1\wedge d\eta_1+\eta_2\wedge d\eta_2+\eta_3\wedge d\eta_3$
that is very well adapted to the $3$-Sasakian structure. 
It is therefore called the \emph{canonical $G_2$ structure} of the underlying
$3$-Sasakian structure. Corollary \ref{cor:g2explkilling} ensures then
the existence of a $\bar\nabla$-parallel spinor field on $\bar M$.

 We calculate the class of the $\Spin(7)$ structure on $\bar M$
of the canonical $G_2$ structure  using  Lemma \ref{lem:spin7str}. 
\begin{thm}
The $\Spin(7)$ structure on the  cone constructed from the canonical 
$G_2$ structure of  a $3$-Sasakian manifold is of class $\mathcal U_1$
if and only if the cone constant is $a=\frac{15}{14}$.
\end{thm}
\begin{proof}
Due to the formulation of the second condition of Lemma 
\ref{lem:spin7str} given in Lemma \ref{lem:seccondlem}, we just need to 
calculate $*\phi$ and $\nabla^g\phi$. Obviously $*\phi$ is given by
\bdm
*\phi=\eta_{4567}-\eta_{2367}-\eta_{2345}-\eta_{1357}+\eta_{1346}
-\eta_{1256}-\eta_{1247}.
\edm
To get $\nabla^g\phi$ we observe
\begin{align*}
\nabla^g_{e_j}\phi=\ &(\nabla^g_{e_j}\eta_1)\wedge F_1
+(\nabla^g_{e_j}\eta_2)\wedge F_2+(\nabla^g_{e_j}\eta_3)\wedge F_3\\
&+\eta_1\wedge(\nabla^g_{e_j}F_1)+\eta_2\wedge(\nabla^g_{e_j}F_2)
+\eta_3\wedge(\nabla^g_{e_j}F_3)\\
&+4(\nabla^g_{e_j}\eta_1)\wedge\eta_2\wedge\eta_3
+4\eta_1\wedge(\nabla^g_{e_j}\eta_2)\wedge\eta_3
+4\eta_1\wedge\eta_2\wedge(\nabla^g_{e_j}\eta_3)
\end{align*}
and since $(\eta_i,F_i)$ are Sasakian structures we have 
$(\nabla^g_{e_j}F_i)(Y,Z)=g(e_j,Z)\eta_i(Y)-g(e_j,Y)\eta_i(Z)$. Thus 
$\nabla^g_{e_j}F_i=\eta_j\wedge\eta_i$ for $i=1,2,3$ and $j=1,..,7$ 
implying $\eta_i\wedge(\nabla^g_{e_j}F_i)=0$. Since 
\bdm
(\nabla^g_X\eta_i)Y=g(Y,\nabla^g_X\xi_i)=g(Y,-\phi_iX)=F_i(X,Y)
\edm
we have $\nabla^g_X\eta_i=X\lrcorner F_i$ and get
\begin{align*}
\nabla^g_{e_j}\phi=&(e_j\lrcorner F_1)\wedge F_1
+(e_j\lrcorner F_2)\wedge F_2+(e_j\lrcorner F_3)\wedge F_3\\
&+4(e_j\lrcorner F_1)\wedge\eta_2\wedge\eta_3
+4\eta_1\wedge(e_j\lrcorner F_2)\wedge\eta_3+4\eta_1\wedge\eta_2\wedge(e_j\lrcorner F_3).
\end{align*}
This gives us
\begin{alignat*}{2}
 \nabla^g_{e_1}\phi\ =\ &-\eta_{346}+\eta_{357}+\eta_{247}+\eta_{256},& \quad
 \nabla^g_{e_2}\phi\ =\ &\eta_{345}+\eta_{367}-\eta_{147}-\eta_{156},\\
 \nabla^g_{e_3}\phi\ =\ &-\eta_{245}-\eta_{267}+\eta_{146}-\eta_{157},& \quad
 \nabla^g_{e_4}\phi\ =\ &3\,(-\eta_{235}+\eta_{567}+\eta_{136}-\eta_{127}),\\
 \nabla^g_{e_5}\phi\ =\ &3\,(\eta_{234}-\eta_{467}-\eta_{137}-\eta_{126}),& \quad
 \nabla^g_{e_6}\phi\ =\ &3\,(-\eta_{237}+\eta_{457}-\eta_{134}+\eta_{125}),\\
 \nabla^g_{e_7}\phi\ =\ &3\,(\eta_{236}-\eta_{456}+\eta_{135}+\eta_{124}). & &
\end{alignat*}
Using an appropriate computer algebra system we easily calculate
\bdm
g(\nabla^g\phi,*\phi)=180,~~~ g(*\phi,*\phi)=168,
\edm
thus the first condition of Lemma \ref{lem:spin7str} is satisfied if 
$a=\frac{15}{14}$. Using the formulation given in Lemma \ref{lem:seccondlem} 
of the second condition one easily checks that the this condition is satisfied for any $a$.
\end{proof}
\end{exa}
We expect that for all other values of the cone constant $a$, the structure
is of generic class $\mathcal U_1\oplus U_2$, but the system of equations that
one obtains is extremely involved. 

\vspace{2cm}
\end{document}